\documentclass[10pt,oneside]{article}
\usepackage{amsmath}
\usepackage{amssymb}
\usepackage{amsfonts}
\usepackage{mathrsfs}
\renewcommand{\Bbb}{\mathbb}
\newcommand{\RR}{\Bbb{R}}
\newcommand{\NN}{\Bbb{N}}

\newcommand{\CC}{\Bbb{C}}
\newcommand{\eps}{\varepsilon}
\DeclareMathOperator{\dive}{div}
\providecommand{\keywords}[1]{\textbf{{Keywords.}} #1}
\providecommand{\AMS}[1]{\textbf{{AMS subject classifications.}} #1}
\renewcommand{\abstract}[1]{\textbf{{Abstract.}} #1}
\newenvironment{proof}{\paragraph{Proof:}}{$\square$}
\newtheorem{theorem}{Theorem}[section]
\newtheorem{proposition}{Proposition}[section]
\newtheorem{lemma}{Lemma}[section]
\newtheorem{remark}{Remark}[section]
\newtheorem{assumption}{Assumption}[section]
\evensidemargin=\oddsidemargin
\usepackage{scalerel,stackengine}
\stackMath
\newcommand\reallywidehat[1]{%
\savestack{\tmpbox}{\stretchto{%
  \scaleto{%
    \scalerel*[\widthof{\ensuremath{#1}}]{\kern-.6pt\bigwedge\kern-.6pt}%
    {\rule[-\textheight/2]{1ex}{\textheight}}
  }{\textheight}%
}{0.5ex}}%
\stackon[1pt]{#1}{\tmpbox}%
}
\begin{document}
\title{Well-posedness of a time-harmonic elasticity problem in a half-strip}
\author{Jean-Luc AKIAN\thanks{ONERA, Universit\'{e} Paris Saclay F-92322 Ch\^{a}tillon, France, (jean-luc.akian@onera.fr).}}         
\maketitle

\abstract{
In this paper we establish that the time-harmonic elasticity problem in a half-strip with non-homogeneous Dirichlet conditions on its boundary section and traction-free conditions on its upper and lower boundaries, has a unique weak solution when this solution is searched under the form of the sum of a linear combination of outgoing waves and an exponentialy decreasing function (in the sense of weighted Sobolev spaces).}


\AMS{74B05,74H20,74H25,74J05,47A53,46E35}
\section{Introduction.}
\label{intro}
\setcounter{equation}{0}
The diffraction of elastic waves  by a defect in a plate (for example in the context of non-destructive testing) can be modelled using Lamb modes. Lamb modes are obtained by searching solutions to the time-harmonic elasticity equations in a plate assuming an invariance with respect to the direction in the plate orthogonal to the direction of wave propagation, and then the problem in the plate is reduced to a problem in a strip. But a fundamental issue (among others) is to reduce the problem of diffraction by a defect which is posed in an unbounded strip to a problem posed in a bounded strip. For that one must find a framework where
the time-harmonic elasticity problem in a half-strip with non-homogeneous Dirichlet conditions on its boundary section is well-posed.
If there exists such a framework, then it is possible to define a Dirichlet-to-Neumann operator on a boundary section of a bounded strip containing the defect (by continuity of the displacement and the traction on this section), in order to prescribe appropriate boundary conditions on it, and thus to reduce the problem to a problem posed in a bounded strip.
It is the purpose of the present paper to find such a framework and to prove a result of well-posedness.

More precisely we consider the time-harmonic elasticity problem in a half-strip with non-homogeneous Dirichlet conditions on its boundary section, traction-free conditions on its upper and lower boundaries and appropriate radiation conditions. 
The radiation conditions consist in searching the solution under the form of the sum of a linear combination of outgoing waves and an exponentialy decreasing function (in the sense of weighted Sobolev spaces). The main theorem of this paper is Theorem \ref{th-half-plane} which establishes the existence and uniqueness of a weak solution (namely in the variational sense) to the equations of time-harmonic elasticity in a half-strip satisfying the aforementioned boundary conditions and radiation conditions. The author was unable to find a complete proof of the well-posedness of this problem in the literature, despite the simplicity of its statement. The uniqueness of a solution to it has been proved for example in \cite{Nazarov1}, Theorem 1 or in \cite{Bouhennache}, Theorem 2.1.  The present paper proves the existence of a solution to this problem. The proofs of the present paper are based on the methods of the books \cite{Nazarov-Plamenevsky}, \cite {Kozlov-Mazya-Rossmann} and of the paper \cite{Bourgeois-Chesnel-Fliss}. However, the framework in the present paper which is the variational framework is not the same as that of references \cite{Nazarov-Plamenevsky}, \cite {Kozlov-Mazya-Rossmann} and we give here {\it detailed proofs} with {\it precise references} to justify the assertions. Note that in \cite{Nazarov}, Theorem 5.1, a result linked with Theorem \ref{th-half-plane} is stated and for the proof the reader is referred to the book \cite{Nazarov-Plamenevsky}. But in \cite{Nazarov-Plamenevsky} it is impossible to find the precise result \cite{Nazarov}, Theorem 5.1. In \cite{Nazarov-Plamenevsky} the theorem which is closer to Theorem 5.1 of \cite{Nazarov} is Theorem 3.5, p.160 and it is not obvious at all that this theorem implies Theorem 5.1 of \cite{Nazarov}.

Essential tools of the proofs in this paper are the use of weighted Sobolev spaces in order to take into account the exponential growth or decay of the solutions at infinity, the fundamental Proposition \ref{pro2sol} which gives an asymptotic expansion of the solution to the time-harmonic elasticity problem in a strip with respect to the eigenvectors and generalized eigenvectors of the corresponding spectral problem (Lamb modes and associated modes are a part of these eigenvectors), the complex symplectic form \eqref{eq15n11} which provides a way to distinguish outgoing and incoming waves and the Fredholm property of some operators. 
The paper is basically {\it self-contained} and is the opportunity to give {\it complete}, and in some instances, new or simpler proofs of rather classical results.

The paper is organized as follows. In section \ref{Set-up} the systems of equations of the time-harmonic elasticity problems in a strip and in a half-strip are written in a formal form. The weighted Sobolev spaces for a strip are introduced in section \ref{weighted}. In section \ref{strip} we give a sufficient condition in order to establish the well-posedness of the time-harmonic elasticity problem in a strip (Theorem \ref{thisom}). This kind of result is fairly classical (see \cite{Kozlov-Mazya-Rossmann}, Chapter 5, \cite{Nazarov-Plamenevsky}, Chapter 3). But in this paper we consider weak solutions and the given proof is rather an adaptation of the proof of Theorem 5.1 of \cite{Bourgeois-Chesnel-Fliss}. In section \ref{strip-Asymptotics} we establish an asymptotics of the solution to the previous problem (Proposition \ref{pro2sol}). There to this kind of result is quite classical (see \cite{Kozlov-Mazya-Rossmann}, Chapter 5, \cite{Nazarov-Plamenevsky}, Chapter 3) but here we consider weak solutions and the proofs must be adapted. Moreover it should be noted that the proof of the corresponding result in \cite{Kozlov-Mazya-Rossmann}, Chapter 5 is not correct as it is and that the calculus of the coefficients of the asymptotics (formula \eqref{eq15l1}) is more straightforward in the present proof than in the aforementioned references (see Remark \ref{rem_error}).
Section \ref{Fredholm} is devoted to the Fredholm property of the operator (Eq. \eqref{propfredholm}) associated to the variational formulation of the half-strip problem (Proposition \ref{propfredholm}). For that purpose an essential tool is Proposition \ref{prop1}. The method of proof of this proposition is similar to that of Theorem 5.3 of \cite{Bourgeois-Chesnel-Fliss}. However the boundary conditions are not the same on the upper and lower parts of the strip (Neumann in the present paper versus Dirichlet in \cite{Bourgeois-Chesnel-Fliss}), and the proof of Proposition \ref{prop1} relies on a precise form of Korn inequality (Theorem \ref{thKorn}), which itself is an application of Theorem 3.1 of \cite{Grabovsky} (Theorem \ref{thGrab}).
In section \ref{Asymptotics} we establish an asymptotics of a solution to the half-strip problem (Proposition \ref{asymt}). In section \ref{Radiation} we introduce the sesquilinear antihermitian form $q_{\Omega_+}$ (Eq. \eqref{eq15n11}). In Proposition \ref{prosympl} we show that this form is nondegenerate so that it is symplectic. This form makes it possible to distinguish outgoing and incoming waves according to the Mandelstam radiation principle.
In section \ref{well-posedness} we show that the half-strip problem is well-posed (Theorem \ref{th-half-plane}) under Assumption \ref{ass1}. In appendix \ref{oppenc} we recall classical results on operator pencils, in appendix \ref{results} we establish lemmas showing the linear independence of some families of functions and in appendix \ref{alg} we give a new proof of an important result but in the variational framework.
\section{Set-up of the problem.}
\label{Set-up}
\setcounter{equation}{0}
In the sequel we shall use the following notations. The set of natural numbers will be  denoted by $\NN$ (containing $0$) and the set of positive natural numbers by $\NN^*$ (= $\NN \setminus \{0\}$).
If $E$ is a Banach space, denote by $E^*$ its topological dual, $\langle .,. \rangle$ the duality bracket between $E^*$ and $E$ (the context will determine the involved Banach space) and if $f \in E^*$ its norm is defined by $\lVert f \rVert_{E^*}$ = 
$\underset{u \in E, \lVert u \rVert \leq 1}{\mbox{sup}} \lvert\langle f,u \rangle \rvert$. 
If $n, m \in \NN^*$ and $\Omega$ is an open set of $\RR^n$, the set of ${\cal C}^{\infty}$ functions from $\Omega$ (resp. $\overline{\Omega}$) with values in $\RR^m$ will be denoted by ${\cal C}^{\infty}(\Omega, \RR^m)$ (resp. ${\cal C}^{\infty}(\overline{\Omega}, \RR^m)$), with similar notations for functions with values in $\CC^m$ and for functions with compact support in $\Omega$ (resp. $\overline{\Omega}$), where ${\cal C}^{\infty}$ is replaced by ${\cal C}^{\infty}_0$. If $k \in \NN$, the set of functions from $\Omega$ with values in  $\RR^m$ whose components are in the Sobolev space $H^k(\Omega)$ will be denoted by $H^k(\Omega,\RR^m)$, with similar notations for functions with values in $\CC^m$ (if $k=0$, $H^0$ is replaced by $L^2$).
The inner product in $H^k(\Omega,\RR^m)$ or $H^k(\Omega,\CC^m)$ will be denoted by ${(.,.)_{k,\Omega}}$, the associated norm by ${||.||_{k,\Omega}}$, the associated semi-norm by ${|.|_{k,\Omega}}$ and the norm in the topological dual $(H^k(\Omega,\RR^m))^*$ or $(H^k(\Omega,\CC^m))^*$ by ${||.||_{k,\Omega,*}}$.
Recall that the word "iff" means "if and only if". We shall use the Einstein summation convention of repeated indices.

Let us consider a linearly elastic plate of thickness $2h$ ($h>0$) occupying the open set $\Omega_1$ of $\RR^3$:
\begin{equation}
\label{eq1a}
\Omega_1 = \{x =(x_1,x_2,x_3) \in \RR^3, \, x_1 \in \omega_h \}
\end{equation}
where $\omega_h = (-h,h)$.
The plate is assumed to be homogeneous and isotropic (with Lam\'{e} coefficients $\lambda$ and $\mu$ satisfying the conditions $3 \lambda + 2 \mu >0$ and $\mu >0$, see \cite{Salencon}, p.341), with mass density $\rho$ and traction-free on the upper and lower boundary. Denote by $u=(u_i)$, $\eps_{ij}(u)$, $\sigma_{ij}(u)$ the displacement field (with values in $\CC^3$), the components of the strain tensor and the components of the stress tensor associated to $u$. 
We assume that the displacement field $u$ is independent of $x_2$ (the physical properties are independent of the $x_2$ direction).
Let us define the strip $\Omega$ by:
\begin{equation}
\label{eq1ax}
\Omega = \{x =(x_1,x_3) \in \RR^2, \, x_1 \in \omega_h \}.
\end{equation}
The time-harmonic elastodynamics equations and boundary conditions for the plate are reduced to a problem in the strip $\Omega$ and are written as follows {(the derivative with respect to $x_i$ being denoted by $\partial_i$):}
\begin{equation}
\label{eq2x}
\partial_j\sigma_{ij}(u) + \rho \omega^2 {u_i} = - f_i \mbox{  in } \Omega
\end{equation}
with
\begin{equation}
\label{eq4x}
\sigma_{ij}(u) = \lambda {(\dive u)} \delta_{i,j} + 2 \mu \eps_{ij}(u) \mbox{  in } \Omega,
\end{equation}
\begin{equation}
\label{eq4ax}
\eps_{ij}(u) = \frac{1}{2}\left(\partial_{j} u_i + \partial_{i} u_j \right) \mbox{  in } \Omega
\end{equation}
and
\begin{equation}
\label{eq4bx}
\sigma_{i1}(u)(x_1 = \pm h) = 0,
\end{equation}
where $\omega >0$ is fixed and $f_i$ are the components of the body force density (independent of $x_2$).
The displacement field ${u}$ (with values in $\CC^3$) formally satisfies the following variational problem: for all displacement field $v \in {\cal C}^{\infty}_0(\overline{\Omega}, \CC^3)$, 
\begin{equation}
\label{eq1x}
b_{\Omega}(u,v)= F(v)
\end{equation}
with the notations
\begin{equation}
\label{eq2a}
b_{\Omega}(u,v) = \int _{\Omega} \sigma_{ij}(u) \eps_{ij}(\overline{v}) - \omega^2 \int _{\Omega} \rho u_i \overline{v_i}
\end{equation}
and
\begin{equation}
\label{eq3x1}
F(v)= \int _{\Omega} f_i \overline{v_i}.
\end{equation}
Now we are interested in the time-harmonic elastodynamics equations but in the half-strip $x_3>0$, with no body force and with a prescribed displacement field $g$ on the boundary $\{x_3 = 0\}$.
Let us set:
\begin{equation}
\label{eq1ay}
\Omega_+ = \left \{ x =(x_1,x_3) \in \RR^2, x_1 \in \omega_h, x_3 >0 \right \},\, 
\Sigma = \left \{ x =(x_1,x_3) \in \RR^2, x_1 \in \omega_h, x_3 = 0 \right \}.
\end{equation}
The time-harmonic elastodynamics equations and boundary conditions for the half-strip $\Omega_+$ are written as follows
\begin{equation}
\label{eq2x1}
\partial_j\sigma_{ij}(u) + \rho \omega^2 {u_i} = 0 \mbox{  in } \Omega_+
\end{equation}
with
\begin{equation}
\label{eq4x1}
\sigma_{ij}(u) = \lambda {(\dive u)} \delta_{i,j} + 2 \mu \eps_{ij}(u) \mbox{  in } \Omega_+,
\end{equation}
\begin{equation}
\label{eq4ax1}
\eps_{ij}(u) = \frac{1}{2}\left(\partial_{j} u_i + \partial_{i} u_j \right) \mbox{  in } \Omega_+,
\end{equation}
\begin{equation}
\label{eq4bx1}
\sigma_{i1}(u)(x_1 = \pm h,\, x_3 >0) = 0
\end{equation}
and
\begin{equation}
\label{eq4bx2}
u \vert_{\Sigma}=g.
\end{equation}
The displacement field ${u}$ (with values in $\CC^3$) formally satisfies the following variational problem: $u$ satisfies \eqref{eq4bx2} and for all displacement field $v \in {\cal C}^{\infty}_0(\overline{\Omega}_+, \CC^3)$ satisfying \eqref{eq4bx2} with $g=0$, 
\begin{equation}
\label{eq1x1}
b_{\Omega_+}(u,v) = 0
\end{equation}
with the notation
\begin{equation}
\label{eq2a1}
b_{\Omega_+}(u,v) = \int _{\Omega_+} \sigma_{ij}(u) \eps_{ij}(\overline{v}) - \omega^2 \int _{\Omega_+} \rho u_i \overline{v_i}.
\end{equation}
\section{Weighted Sobolev spaces.}
\label{weighted}
\setcounter{equation}{0}
In this section we introduce the weighted Sobolev spaces and the partial Fourier-Laplace transform for a strip.
For $\beta \in \RR$, $k \in \NN$ and $m \in \NN^*$, define $W^k_{\beta}(\Omega, \RR^m)$ as the completion of ${\cal C}^{\infty}_0(\overline{\Omega},\RR^m)$ for the norm $\lVert u \rVert_{\beta,k,\Omega}$ = $\lVert e^{\beta x_3} u \rVert_{k,\Omega}$ with similar definitions for $W^k_{\beta}(\Omega, \CC^m)$.
The norm on the topological dual $(W^k_{\beta}(\Omega, \RR^m))^*$ (resp. $(W^k_{\beta}(\Omega, \CC^m))^*$) of $W^k_{\beta}(\Omega, \RR^m)$ (resp. $W^k_{\beta}(\Omega, \CC^m)$) will be denoted by $||.||_{\beta,k,\Omega,*}$.
Define the partial Fourier-Laplace transform with respect to $x_3$  on ${\cal C}^{\infty}_0(\overline{\Omega}, \CC^m)$ by

\begin{equation}
\label{eq100}
\hat{v}(.,\nu) = ({\cal L}_{x \rightarrow \nu} v) (.,\nu) := \int_{-  \infty}^{+\infty} e^{- \nu x} v(.,x)dx,\, v \in {\cal C}^{\infty}_0(\overline{\Omega}, \CC^m),\, \nu \in \CC.
\end{equation}
It is continuously extended to an isomorphism between $W^k_{\beta}(\Omega, \CC^m)$ and 
\begin{equation}
\label{eq101}
\reallywidehat{W^k_{\beta}(\Omega,\CC^m)}= \left \{\hat{v} \in L^2(l_{- \beta}, H^k(\omega_h,\CC^m)), \, \frac{1}{i}\int_{l_{-\beta}} \lVert \hat{v}(.,\nu) \rVert^2_{k,\omega_h, |\nu|} d \nu< +\infty \right \}
\end{equation}
equipped with the norm:
\begin{equation}
\label{eq101a}
\frac{1}{2 \pi i}\int_{l_{-\beta}} \lVert \hat{v}(.,\nu) \rVert^2_{k,\omega_h, |\nu|} d \nu,  \hat{v} \in \reallywidehat{W^k_{\beta}(\Omega,\CC^m)},
\end{equation}
where for $\beta \in \RR$, $l_{\beta} = \beta + i \RR$, and $||.||_{k,\omega_h, |\nu|}$ is the norm on $H^k(\omega_h,\CC^m)$ defined by:
\begin{equation}
\label{eq102}
\lVert \varphi \rVert _{k,\omega_h, |\nu|} = \left( \sum_{\alpha,\gamma \in \NN, \alpha + \gamma \leq k} \lVert \lvert \nu \rvert^{\alpha} \partial_{x_1}^{\gamma} \varphi\rVert_{0,\omega_h}^2 \right)^{1/2},\, \varphi \in H^k(\omega_h,\CC^m).
\end{equation}
For $\nu = -\beta + i s$ $\in l_{-\beta}$, we have
\begin{equation}
\label{eq102a}
\hat{v}(.,\nu) =  F( e^{\beta .} v) (.,s),\,  v \in W^k_{\beta}(\Omega, \CC^m),
\end{equation}
where $F$ is the partial Fourier transform defined by:
\begin{equation}
\label{eq102b}
F{u}(.,s) = \int_{-  \infty}^{+\infty} e^{- i s x} u(.,x) dx,\,  u \in L^2(\Omega,\CC^m).
\end{equation}
The inverse partial Fourier-Laplace transform ${\cal L}_{x \rightarrow \nu}^{-1}$ is given by:
\begin{equation}
\label{eq104}
 v(.,x) = ({\cal L}^{-1}_{x \rightarrow \nu} \hat{v})(.,x) = \frac{1}{2 \pi i} \int_{l_{-\beta}} e^{\nu x} \hat{v}(.,\nu) d \nu, \,  \hat{v} \in \reallywidehat{{W^0_{\beta}(\Omega,\CC^m)}},\,  x \in \RR.
\end{equation}
The topological dual of $\reallywidehat{W^k_{\beta}(\Omega,\CC^m)}$ is:
\begin{equation}
\label{eq105}
(\reallywidehat{W^k_{\beta}(\Omega,\CC^m)})^*= \left \{\hat{g} \in L^2(l_{\beta}, (H^k(\omega_h,\CC^m))^*), \, \frac{1}{i}\int_{l_{\beta}} \lVert \hat{g}(.,\nu) \rVert^2_{k,\omega_h,|\nu|,*} d \nu < +\infty \right \}
\end{equation}
equipped with the norm:
\begin{equation}
\label{eq105a}
\frac{1}{2 \pi i}\int_{l_{\beta}} \lVert \hat{g}(.,\nu) \rVert^2_{k,\omega_h,|\nu|,*} d \nu,  \hat{g} \in (\reallywidehat{W^1_{\beta}(\Omega,\CC^m)})^*
\end{equation}
where $||.||_{k,\omega_h, |\nu|,*}$ is the norm on the topological dual $(H^k(\omega_h,\CC^m))^*$ of $H^k(\omega_h,\CC^m)$ itself equipped with the norm $||.||_{k,\omega_h,|\nu|}$.
The Parseval formula reads as follows: for all $u \in W^0_{\beta}(\Omega,\CC^m)$, $v \in W^0_{-\beta}(\Omega,\CC^m)$, we have
\begin{equation}
\label{eq108}
\int_{-\infty}^{+\infty} u(., x) \overline{v}(.,x) dx = \frac{1}{2 \pi} \int_{-\infty}^{+\infty} \hat{u}(.,-\beta +i s) \overline{ \hat{v}(.,\beta + i s)} ds,
\end{equation}
or
\begin{equation}
\label{eq109}
\int_{-\infty}^{+\infty} u(., x) \overline{v}(.,x) dx = \frac{1}{2 \pi i} \int_{l_{-\beta}} \hat{u}(.,\nu) { \hat{\overline{v}}(.,-\nu)} d \nu,
\end{equation}
and since 
\begin{equation}
\label{eq109a}
 \hat{u}(.,\nu) = F(e^{\beta .} u)(.,s),\, \hat{\overline{v}}(.,-\nu) = \overline{F(e^{-\beta .} v)(.,s)},\,\nu = -\beta +i s \in l_{-\beta},
\end{equation}
we have also
\begin{equation}
\label{eq109b}
\int_{-\infty}^{+\infty} u(., x) \overline{v}(.,x) dx = \frac{1}{2 \pi} \int_{- \infty}^{+ \infty}F(e^{\beta .} u)(.,s)
 \overline{F(e^{-\beta .} v)(.,s)} ds.
\end{equation}
On the other hand,
\begin{equation}
\label{eq110}
 \widehat{\partial_x u}(.,\nu) = \nu  \hat{u}(.,\nu),\,  u \in W^1_{\beta}(\Omega,\CC^m),\, \nu \in l_{-\beta}.
\end{equation}
The partial Fourier-Laplace transform ${\cal L}_{x \rightarrow \nu} f$ =  $\hat{f}$ of $f \in ({W^k_{\beta}(\Omega,\CC^m)})^*$ is characterized by $\hat{f} \in (\reallywidehat{W^k_{\beta}(\Omega,\CC^m)})^*$ and 
\begin{equation}
\label{eq106}
\frac{1}{2 \pi i} \int_{l_{\beta}} \langle \hat{f}(., \nu), \hat{v}(.,- \nu) \rangle d \nu = 
\langle f, v \rangle,\, v \in  W^k_{\beta}(\Omega,\CC^m)
\end{equation}
which is consistent with \eqref{eq109}.
\section{Well-posedness of the problem in the strip $\Omega$.}
\label{strip}
\setcounter{equation}{0}
In this section we shall study the mathematical formulation of the variational problem \eqref{eq1x}, \eqref{eq2a}, \eqref{eq3x1}.
Let $\beta \in \RR$. The sesquilinear form $b_{\Omega}$ defined in \eqref{eq2a} makes sense for $u \in W^1_{\beta}(\Omega,\CC^3)$ and  $v \in W^1_{-\beta}(\Omega,\CC^3)$ and there exists $C>0$ satisfying:
\begin{equation}
\label{eq3}
|b_{\Omega}(u,v)| \leq C||u||_{\beta,1,\Omega} ||v||_{-\beta,1,\Omega},\,  u \in W^1_{\beta}(\Omega,\CC^3), \,  v \in W^1_{-\beta}(\Omega,\CC^3),
\end{equation}
therefore one can define a linear and continuous map ${B}_{\Omega,\beta}$: $W^1_{\beta}(\Omega,\CC^3)$ $\rightarrow$ $({W}^1_{-\beta}(\Omega,\CC^3))^*$ by
\begin{equation}
\label{eq4}
<B_{\Omega, \beta}u,v> = b_{\Omega}(u,\overline{v}),\,  u \in W^1_{\beta}(\Omega,\CC^3), v \in W^1_{-\beta}(\Omega,\CC^3),
\end{equation}
and we have
\begin{equation}
\label{eq5}
||B_{\Omega,\beta}u||_{-\beta,1,\Omega,*} \leq C ||u||_{\beta,1,\Omega},\,  u \in {W^1_{\beta}(\Omega,\CC^3)} .
\end{equation}
The next three lemmas are used in the proof of Theorem \ref{thisom}.
\begin{lemma}
\label{lap}
If $\nu \in \CC$, define the linear continuous map $\mathscr{L}(\nu)$: $H^1(\omega_h,\CC^3)$ $\rightarrow$  $(H^1(\omega_h,\CC^3))^*$ by:
\begin{equation}
\label{eq4b}
\langle  \mathscr{L}(\nu) \varphi, \psi \rangle =  l(\nu)(\varphi, \overline{\psi}),
\, \varphi, \psi \in H^1(\omega_h, \CC^3),
\end{equation}
where for $\nu \in \CC$, $l(\nu)$ is defined by
\begin{equation}
\label{eq13}
 l(\nu)(\varphi, \psi) = a(\varphi, \psi) + (-i \nu) b(\varphi, \psi) + (-i \nu)^2 c(\varphi, \psi),\, \varphi, \psi \in H^1(\omega_h, \CC^3)
\end{equation}
and the sesquilinear Hermitian forms $a$,$b$,$c$ are defined by: for $\varphi, \psi$ $\in H^1(\omega_h, \CC^3)$,
\begin{equation}
\label{eq13a}
a(\varphi, \psi) = a_0(\varphi, \psi) - \omega^2 m(\varphi, \psi),
\end{equation}
\begin{equation}
\label{eq14}
a_0(\varphi, \psi) = \int_{\omega_h} (\lambda + 2 \mu) \partial_1 \varphi_1 \partial_1\overline{\psi_1} + \mu(\partial_1 \varphi_2 \partial_1\overline{\psi_2} + \partial_1 \varphi_3 \partial_1\overline{\psi_3}),
\end{equation}
\begin{equation}
\label{eq14a}
m(\varphi, \psi) = \int_{\omega_h} \rho (\varphi_1 \overline{\psi_1}+ \varphi_2 \overline{\psi_2}+ \phi_3 \overline{\psi_3}),
\end{equation}
\begin{equation}
\label{eq15}
b(\varphi, \psi) = \int_{\omega_h} \lambda (-i \partial_1 \varphi_1 \overline{\psi_3} + i \varphi_3  \partial_1 \overline {\psi_1}) + \mu (-i \partial_1 \varphi_3 \overline{\psi_1} + i \varphi_1 \partial_1 \overline{\psi_3}),
\end{equation}
\begin{equation}
\label{eq16}
c(\varphi, \psi) = \int_{\omega_h} \mu (\varphi_1 \overline{\psi_1} + \varphi_2 \overline{\psi_2}) + (\lambda + 2 \mu)  \varphi_3 \overline{\psi_3}.
\end{equation}
Then if $u \in W^1_{\beta}(\Omega,\CC^3)$ and $v \in W^1_{-\beta}(\Omega,\CC^3)$,
\begin{eqnarray}
\label{eq4c}
<B_{\Omega,\beta}u,v>  &=& b_{\Omega}(u, \overline{v}) = \displaystyle \frac{1}{2 \pi i} \int_{l_{-\beta}} {l}(\nu) (\hat{u}(.,\nu),\overline{\hat{v}(.,-\nu)}) d \nu\nonumber \\
&=& \displaystyle \frac{1}{2 \pi i} \int_{l_{-\beta}} 
\langle  \mathscr{L}(\nu) \hat{u}(.,\nu),{\hat{v}(.,-\nu)} \rangle d \nu.
\end{eqnarray}
\end{lemma}
\begin{proof}
If $u \in W^1_{\beta}(\Omega,\CC^3)$ and $v \in W^1_{-\beta}(\Omega,\CC^3)$, 
\begin{eqnarray}
\label{eq5x}
\sigma_{ij}({u}) \eps_{ij}(\overline{v}) &= &\lambda {(\dive {u})} {(\dive \overline{v})} + 2 \mu \eps_{ij}({u}) \eps_{ij}(\overline{v}) \nonumber \\
&= &\lambda {(\partial_1 u_1 + \partial_3 u_3) (\partial_1 \overline{v}_1 + \partial_3 \overline{v}_3)}\nonumber \\
&+ & \mu \{ 2 \partial_1 u_1 \partial_1 \overline{v}_1 + 2 \partial_3 u_3 \partial_3 \overline{v}_3 \nonumber \\
&+ & \partial_1 u_2 \partial_1 \overline{v}_2
+(\partial_1 u_3 + \partial_3 u_1) (\partial_1 \overline{v}_3 + \partial_3 \overline{v}_1)+ \partial_3 u_2 \partial_3 \overline{v}_2 \}
\end{eqnarray}
and if $\nu \in l_{-\beta}$, 
\begin{equation}
\label{eq7}
\widehat{{\dive {u}}} = 
\partial_1 \hat{u}_1 + \nu \hat{u}_3,
\end{equation}
\begin{equation}
\label{eq8}
\widehat{\eps_{11} ({u})} = \partial_1 \hat{u}_1 ,\, \widehat{\eps_{22} ({u})} =  0, \, \widehat{ \eps_{33} ({u})} = \nu \hat{u}_3,
\end{equation}
\begin{equation}
\label{eq9}
\widehat{\eps_{12} ({u})} = 1/2 \partial_1 \hat{u}_2,\, \widehat{\eps_{13} ({u})} = 1/2 (\partial_1 \hat{u}_3 +\nu \hat{u}_1), \, \widehat{\eps_{23} ({u})} = 1/2 \nu \hat{u}_2,
\end{equation}
consequently in view of \eqref{eq109} we get \eqref{eq4c}.
\end{proof}
%

%
\begin{lemma}
\label{isom}
There exist $s_0 >0$ and $\alpha>0$ such that if $\nu=t+ is$, $|s| \geq s_0$ and $|t| \leq \alpha |s|$, 
$\mathscr{L}(\nu)$ is an isomorphism from 
$H^1(\omega_h, \CC^3)$ onto $(H^1(\omega_h, \CC^3))^*$. Moreover there exists $C>0$ such that for these values of $\nu$,
\begin{equation}
\label{eq103h}
||\varphi||_{1,\omega_h,|\nu|} \leq C ||\mathscr{L}(\nu) \varphi||_{1,\omega_h,|\nu|,*},  \varphi \in H^1(\omega_h, \CC^3).
\end{equation}
\end{lemma}
%
\begin{proof}
We have
\begin{equation}
\label{eq13b}
 l(i s)(\varphi, \psi) = a(\varphi, \psi) + s b(\varphi, \psi) + s^2 c(\varphi, \psi),\, \varphi, \psi \in H^1(\omega_h, \CC^3),\, s \in \RR.
\end{equation}
But with the notation $\widetilde{\Omega} = \omega_h \times (0,1)$
\begin{equation}
\label{eq30}
a(\varphi,\varphi) + s b(\varphi,\varphi) + s^2 c(\varphi,\varphi)= \int_{\widetilde{\Omega}} (\sigma_{ij}(u)\eps_{ij}(\overline{u}) - \omega^2 \rho u_i \overline{u_i}),\, \varphi \in H^1(\omega_h,\CC^3),\, s \in \RR
\end{equation}
where 
\begin{equation}
\label{eq30a}
u(x_1,x_3) = \varphi(x_1) e^{i s x_3},\, (x_1,x_3) \in \widetilde{\Omega}.
\end{equation}
Owing to Korn inequality, there exist $C_1,\, C_2 >0$ satisfying
\begin{equation}
\label{eq31}
\int_{\widetilde{\Omega}} \sigma_{ij}(u)\eps_{ij}(\overline{u}) \geq
C_1 ||u||_{1,\widetilde{\Omega}}^2 -C_2 ||u||_{0,\widetilde{\Omega}}^2,\, u \in H^1(\widetilde{\Omega},\CC^3).
\end{equation}
But if $u$ is given by \eqref{eq30a},
\begin{equation}
\label{eq32}
|u|_{1,\widetilde{\Omega}}^2 = |\varphi|_{1,\omega_h}^2 + s^2 ||\varphi||_{0,\omega_h}^2
\end{equation}
and
\begin{equation}
\label{eq33}
||u||_{0,\widetilde{\Omega}}^2 =||\varphi||_{0,\omega_h}^2.
\end{equation}
Owing to \eqref{eq30}, \eqref{eq31}, \eqref{eq32} and \eqref{eq33}, there exist $C > 0$ and $s_0 \geq 1$ such that if $s \in \RR$, $|s| \geq s_0$ and  $\varphi \in H^1(\omega_h,\CC^3)$,
\begin{eqnarray}
\label{eq35}
a(\varphi,\varphi) + s b(\varphi,\varphi) + s^2 c(\varphi,\varphi) &\geq& C (|\varphi|_{1,\omega_h}^2 + s^2 ||\varphi||_{0,\omega_h}^2) \nonumber \\
&=& C ||\varphi||_{1,\omega_h,|s|}^2 
\geq C ||\varphi||_{1,\omega_h}^2 .
\end{eqnarray}
Let us now examine the behavior of $\mathscr{L}(\nu)$ for $\nu \in \CC$. If $\nu \in \CC$, $\nu = t+ i s$, then for $\varphi \in H^1(\omega_h,\CC^3)$,
\begin{eqnarray}
\label{eq103a}
&&{\mbox {Re}}\left\{a(\varphi,\varphi) + (-i \nu) b(\varphi,\varphi) + (-i \nu)^2 c(\varphi,\varphi)\right\}  = \nonumber \\ 
&&a(\varphi,\varphi) + s b(\varphi,\varphi) + (s^2 - t^2) c(\varphi,\varphi).
\end{eqnarray}
Let us choose $0 < \alpha <1$. Due to \eqref{eq35}, if $|s| \geq s_0$ and $|t| \leq \alpha |s|$,
\begin{eqnarray}
\label{eq103b}
&&a(\varphi,\varphi) + s b(\varphi,\varphi) + (s^2 - t^2) c(\varphi,\varphi)  \geq \nonumber \\
&&a(\varphi,\varphi) + s b(\varphi,\varphi) + (1- \alpha^2)s^2 c(\varphi,\varphi)  \geq \nonumber \\
&&C (|\varphi|_{1,\omega_h}^2 + s^2 ||\varphi||_{0,\omega_h}^2)- \alpha^2s^2 c(\varphi,\varphi),\,  \varphi \in H^1(\omega_h,\CC^3).
\end{eqnarray}
Thus there exist $\alpha$, $0 < \alpha <1$ and $C>0$ such that if $\nu= t+i s$, $|s| \geq s_0$ and $|t| \leq \alpha |s|$ (so that $|\nu|^2 \leq 2 |s|^2$), then for $\varphi \in H^1(\omega_h,\CC^3)$,
\begin{eqnarray}
\label{eq103c}
&&{\mbox {Re}}\left\{a(\varphi,\varphi) + (-i \nu) b(\varphi,\varphi) + (-i \nu)^2 c(\varphi,\varphi)\right\}
\geq \nonumber \\
&&2C ||\varphi||_{1,\omega_h, |s|}^2
\geq C ||\varphi||_{1,\omega_h, |\nu|}^2
\geq C ||\varphi||_{1,\omega_h}^2.
\end{eqnarray}
In view of  \eqref{eq4b}, \eqref{eq13}, \eqref{eq103c} and Lax-Milgram theorem the lemma follows.
\end{proof}
%
\begin{lemma}
\label{lemma_spectrum}
The polynomial operator pencil $\mathscr{L}(\nu)$ is Fredholm and therefore the spectrum of $\mathscr{L}(\nu)$ consists of isolated eigenvalues of finite algebraic multiplicity.
\end{lemma}
%
\begin{proof}
Since the sesquilinear Hermitian forms $a$, $b$, $c$ are continuous on $H^1(\omega_h, \CC^3)$, they define continuous operators $\mathscr{A}$, $\mathscr{B}$, $\mathscr{C}$ from $H^1(\omega_h, \CC^3)$ into $(H^1(\omega_h, \CC^3))^*$ satisfying 
\begin{equation}
\label{eq220}
\langle \mathscr{A} \varphi, \psi \rangle = a(\varphi, \overline{\psi}), \,\langle \mathscr{B} \varphi, \psi \rangle = b(\varphi, \overline{\psi}), \, 
\langle \mathscr{C} \varphi, \psi \rangle = c(\varphi, \overline{\psi}), \,  \varphi, \psi \in H^1(\omega_h, \CC^3)
\end{equation}
and we have 
\begin{equation}
\label{eq221}
\mathscr{L}(\nu) = \mathscr{A} + (-i \nu) \mathscr{B} + (-i \nu)^2 \mathscr{C},  \nu \in \CC.
\end{equation}
On the other hand,
\begin{eqnarray}
\label{eq222}
b(\varphi, \psi) 
&=& i\lambda [\varphi_3 \overline{\psi_1}]_{-h}^h 
+ i\mu [\varphi_1 \overline{\psi_3}]_{-h}^h \nonumber \\
&-& \int_{\omega_h} i(\lambda+ \mu) ((\partial_1 \varphi_1) \overline{\psi_3} 
+(\partial_1\varphi_3) \overline{\psi_1}),\, \varphi, \psi \in H^1(\omega_h, \CC^3).
\end{eqnarray}
There exists $C>0$ such that the last term  in the right hand side of \eqref{eq222} is  majorized by
\begin{equation}
\label{eq224}
C ||\varphi||_{1, \omega_h} ||\psi||_{0, \omega_h},\,  \varphi, \psi \in H^1(\omega_h, \CC^3)
\end{equation}
and thanks to the trace theorems (see \cite{Ciarlet}, p.114) for $p \in \RR$ satisfying $1/2 < p < 1$, there exists $C>0$ such that the first two terms in the right hand side of \eqref{eq222} are majorized by
\begin{equation}
\label{eq223}
C ||\varphi||_{1, \omega_h} ||\psi||_{p, \omega_h},\, \varphi, \psi \in H^1(\omega_h, \CC^3).
\end{equation}
Given that for $p \in \RR$, $1/2 < p < 1$ the embedding from $H^1(\omega_h, \CC^3)$ into $H^p(\omega_h, \CC^3)$ is compact, by duality the embedding from $(H^p(\omega_h, \CC^3))^*$ into $(H^1(\omega_h, \CC^3))^*$ is also compact, and since by \eqref{eq224} and \eqref{eq223} $\mathscr{B}$  is continuous from $H^1(\omega_h, \CC^3)$ into $(H^p(\omega_h, \CC^3))^*$, $\mathscr{B}$ is a compact operator from $H^1(\omega_h, \CC^3)$ into $(H^1(\omega_h, \CC^3))^*$.
It is easily seen that $\mathscr{C}$ is also a compact operator from $H^1(\omega_h, \CC^3)$ into $(H^1(\omega_h, \CC^3))^*$.
Due to Lemma \ref{isom} the polynomial operator pencil ${\cal L}(\nu)$ has at least one regular point. Proposition A.8.4 of \cite{Kozlov-Mazya} recalled in appendix \ref{oppenc} shows the lemma.
\end{proof}
%

It is easily seen that the spectrum of ${\cal L}(\nu)$ is divided in two parts, one concerning the first and the third components of the displacement field (the corresponding eigenvectors and generalized eigenvectors are called Lamb modes and associated modes) and one concerning the second component of the displacement field.

In view of Lemmas \ref{isom} and \ref{lemma_spectrum}, except for a set of $\beta \in \RR$ at most denumerable, $\mathscr{L}(\nu)$ has no eigenvalue on the line $Re \nu = - \beta$.
From the previous lemmas we get
\begin{theorem}
\label{thisom}
Let $\beta \in \RR$ be such that $\mathscr{L}(\nu)$ has no eigenvalue on the line $Re \nu = - \beta$.
Then the operator $B_{\Omega,\beta}$: $W^1_{\beta}(\Omega,\CC^3)$ $\rightarrow$ $({W}^1_{-\beta}(\Omega,\CC^3))^*$ is an isomorphism.
\end{theorem}
%
\begin{proof}
Let $f \in (W_{-\beta}^1 (\Omega,\CC^3))^*$. Let us try to solve 
\begin{equation}
\label{eq103e}
B_{\Omega,\beta}u = f,
\end{equation}
that is: find $u \in W_{\beta}^1 (\Omega,\CC^3)$ such that 
for $v \in W_{-\beta}^1 (\Omega,\CC^3)$,
\begin{equation}
\label{eq103d}
\langle B_{\Omega,\beta}u, v \rangle = b_{\Omega}(u,\overline{v}) = \langle f, v\rangle
\end{equation}
or from \eqref{eq106} and \eqref{eq4c}
\begin{eqnarray}
\label{eq103f}
\displaystyle \frac{1}{2 \pi i} \int_{l_{-\beta}} {l}(\nu) (\hat{u}(.,\nu),\overline{\hat{v}(.,-\nu)}) d \nu &=&
\displaystyle \frac{1}{2 \pi i} \int_{l_{-\beta}} 
\langle  \mathscr{L}(\nu) \hat{u}(.,\nu),\hat{v}(.,-\nu) \rangle d \nu \nonumber \\ 
&=& \frac{1}{2 \pi i} \int_{l_{-\beta}} \langle \hat{f}(., \nu), \hat{v}(.,- \nu) \rangle d \nu
\end{eqnarray}
which is equivalent to
\begin{equation}
\label{eq103g}
\mathscr{L}(\nu) \hat{u}(.,\nu) = \hat{f}(., \nu),  \nu \in l_{-\beta}.
\end{equation}
As $\beta \in \RR$ is such that $\mathscr{L}(\nu)$ has no eigenvalue on the line $Re \nu = - \beta$, for $\nu \in l_{-\beta}$, $\mathscr{L}(\nu)$ is an isomorphism 
from $H^1(\omega_h, \CC^3)$ onto $(H^1(\omega_h, \CC^3))^*$ (Lemma \ref{lemma_spectrum}). Consequently \eqref{eq103g} is equivalent to
\begin{equation}
\label{eq103i}
\hat{u}(.,\nu) = (\mathscr{L}(\nu))^{-1}\hat{f}(., \nu), \nu \in l_{-\beta}.
\end{equation}
Taking into account Lemma \ref{isom} there exists $\nu_{\beta} >0$ such that estimate \eqref{eq103h} is true for $\nu \in l_{-\beta} \setminus [-\beta - i \nu_{\beta}, -\beta + i \nu_{\beta}]$.
On the other hand $(\mathscr{L}(\nu))^{-1}$ is holomorphic in a neighbouhood of $[-\beta - i \nu_{\beta}, -\beta + i \nu_{\beta}]$, then its norm is bounded in this neighbouhood, therefore there exists $C>0$ satisfying
\begin{equation}
\label{eq103j}
||\hat{u}(.,\nu)||_{1,\omega_h,|\nu|} \leq C ||\hat{f}(., \nu)||_{1,\omega_h,|\nu|,*}, \nu \in l_{-\beta}.
\end{equation}
Since $f \in ({W}^1_{-\beta}(\Omega,\CC^3))^*$, from \eqref{eq105} we have
\begin{equation}
\label{eq107a}
\frac{1}{2\pi i} \int_{l_{-\beta}} \lVert \hat{f}(.,\nu) \rVert^2_{1,\omega_h,|\nu|,*} d \nu <+\infty
\end{equation}
accordingly thanks to \eqref{eq101}, \eqref{eq104}, \eqref{eq103j}, \eqref{eq107a} $u$ defined by
\begin{equation}
\label{eq104a}
u(.,x) = ({\cal L}^{-1}_{x \rightarrow \nu} \hat{u})(.,x) = \frac{1}{2 \pi i} \int_{l_{-\beta}} e^{\nu x} \hat{u}(.,\nu) d \nu
= 
 \frac{1}{2 \pi i} \int_{l_{-\beta}} e^{\nu x}(\mathscr{L}(\nu))^{-1}\hat{f}(., \nu)d \nu
\end{equation}
is the unique solution to \eqref{eq103e}. 
\end{proof}

Define the sesquilinear form $\langle .,.\rangle_1$ on $H^1(\omega_h, \CC^3) \times (H^1(\omega_h, \CC^3))^*$ by
\begin{equation}
\label{eq104b}
\langle \varphi,\varphi^*\rangle_1 = \overline{\langle \varphi^*,\overline{\varphi}\rangle},\,  \varphi \in H^1(\omega_h, \CC^3),  \varphi^* \in (H^1(\omega_h, \CC^3))^*.
\end{equation}
This sesquilinear form satisfies conditions \eqref{eq204} and \eqref{eq205} of appendix \ref{oppenc}.

If $A$: $H^1(\omega_h, \CC^3)$ $\rightarrow$ $(H^1(\omega_h, \CC^3))^*$ is one of the operators $\mathscr{A}$,  $\mathscr{B}$, $\mathscr{C}$ of \eqref{eq220}, $A$ satisfies:
\begin{equation}
\label{eq104c}
\langle {A} \psi, \varphi \rangle = \overline{\langle {A} \overline{\varphi}, \overline{\psi} \rangle}, \,  \varphi, \psi \in H^1(\omega_h, \CC^3).
\end{equation}
Let us define a duality between $(H^1(\omega_h, \CC^3))^*$ and $(H^1(\omega_h, \CC^3))^{**}$ (that we will identify with $H^1(\omega_h, \CC^3)$).
Define the sesquilinear form $\langle .,.\rangle_2$ on $(H^1(\omega_h, \CC^3))^* \times (H^1(\omega_h, \CC^3))^{**}$ by: for all $\varphi^* \in (H^1(\omega_h, \CC^3))^*$, $\varphi^{**} \in (H^1(\omega_h, \CC^3))^{**} \cong  H^1(\omega_h, \CC^3)$,
\begin{equation}
\label{eq104d}
\langle \varphi^*,\varphi^{**}\rangle_2 = {\langle \varphi^*,\overline{\varphi^{**}}\rangle} =
\overline{\langle \varphi^{**},\varphi^{*}\rangle_1}.
\end{equation}
This sesquilinear form satisfies the conditions \eqref{eq204} and \eqref{eq205} of appendix \ref{oppenc}.
Let us determine the adjoint in the sense of \eqref{eq206} of an operator $A$: $H^1(\omega_h, \CC^3)$ $\rightarrow$ $(H^1(\omega_h, \CC^3))^*$ satisfying \eqref{eq104c}.
The definition \eqref{eq206} of the adjoint $A^*$ reads as follows in this case:
\begin{equation}
\label{eq104e}
\overline{\langle A^* \psi, \overline{\varphi}\rangle} = \langle A \varphi, \overline{\psi}\rangle,\,  \varphi \in H^1(\omega_h, \CC^3),  \psi \in H^1(\omega_h, \CC^3).
\end{equation}
But \eqref{eq104c} and \eqref{eq104e} imply $A$ = $A^*$ consequently (see \eqref{eq207}):
\begin{equation}
\label{eq104e1}
{\mathscr L}^*(\nu)= {\mathscr L}(-\nu),\,  \nu \in \CC.
\end{equation}
namely
\begin{equation}
\label{eq104e2}
\langle {\mathscr L}(\nu) \varphi, \psi \rangle_2 =\overline{\langle {\mathscr L}(-\overline{\nu}) \psi, \varphi \rangle_2},\,  \nu \in \CC ,\,  \varphi \in H^1(\omega_h, \CC^3),  \psi \in H^1(\omega_h, \CC^3).
\end{equation}
\section{Asymptotics in the strip $\Omega$.}
\label{strip-Asymptotics}
\setcounter{equation}{0}
This section is devoted to the asymptotics of the solution to Eq. \eqref{eq103e} (Proposition \ref{pro2sol}).
Let $\beta \in \RR$. Define the sesquilinear forms $b_1$ and $b_2$ on  $H^1(\omega_h, \CC^3)$ by
\begin{equation}
\label{eq15a}
b_1(\varphi, \psi) = \lambda \int_{\omega_h} \varphi_3  \partial_1 \overline {\psi_1} + \mu \int_{\omega_h} \varphi_1 \partial_1 \overline{\psi_3},\, \varphi, \psi \in H^1(\omega_h, \CC^3),
\end{equation}
\begin{equation}
\label{eq15b}
b_2(\varphi, \psi) = \overline{b_1(\psi, \varphi)} = \lambda \int_{\omega_h} (\partial_1 \varphi_1) \overline {\psi_3} + \mu \int_{\omega_h} (\partial_1\varphi_3) \overline{\psi_1},\, \varphi, \psi \in H^1(\omega_h, \CC^3),
\end{equation}
so that the sesquilinear Hermitian form $b$ defined in \eqref{eq15} is equal to $i(b_1 - b_2)$.
With the notations \eqref{eq2a}, \eqref{eq13a}, \eqref{eq14}, \eqref{eq14a}, \eqref{eq16}, \eqref{eq15a}, \eqref{eq15b}, and owing to \eqref{eq5x}, for $u \in W^1_{\beta}(\Omega,\CC^3)$ and $v \in W^1_{-\beta}(\Omega,\CC^3)$,
\begin{eqnarray}
\label{eq15c}
b_{\Omega}(u,v)&= & \displaystyle{\int_{-\infty}^{+\infty} \{a(u(.,x_3), v(.,x_3))+ b_1(\partial_3 u(.,x_3),v(.,x_3))} \nonumber\\
&& +\displaystyle{b_2(u(.,x_3), \partial_3 v(.,x_3))+ c(\partial_3 u(.,x_3),\partial_3 v(.,x_3))\} dx_3}
.
\end{eqnarray}
If $u \in W^2_{\beta}(\Omega,\CC^3)$, $v \in W^1_{-\beta}(\Omega,\CC^3)$, we have 
\begin{eqnarray}
\label{eq15d}
b_{\Omega}(u,v)&= & \displaystyle{\int_{-\infty}^{+\infty} \{a(u(.,x_3), v(.,x_3))+ b_1(\partial_3 u(.,x_3),v(.,x_3))} \nonumber\\
&& \displaystyle{-b_2(\partial_3 u(.,x_3), v(.,x_3))- c(\partial_{33} u(.,x_3),v(.,x_3))\} dx_3}.
\end{eqnarray}
On the other hand from \eqref{eq4b} and \eqref{eq13}, 
\begin{equation}
\label{eq15e}
\langle  \mathscr{L}(\nu) \varphi, \psi \rangle = 
 a(\varphi, \overline{\psi}) + \nu (b_1(\varphi, \overline{\psi}) - b_2(\varphi, \overline{\psi}))  - \nu^2 c(\varphi, \overline{\psi}),
\, \varphi, \psi \in H^1(\omega_h, \CC^3), \nu \in \CC.
\end{equation}
In view of \eqref{eq15d} and \eqref{eq15e}, if $u \in W^2_{\beta}(\Omega,\CC^3)$, $v \in W^1_{-\beta}(\Omega,\CC^3)$, we have
\begin{equation}
\label{eq4c2}
<B_{\Omega,\beta}u,v>  = b_{\Omega}(u, \overline{v}) 
=
\int_{-\infty}^{+\infty} 
\langle  \mathscr{L}(\partial_x)u(., x),{v}(.,x)\rangle dx.
\end{equation}
We also have: for all $u \in W^2_{\beta}(\Omega,\CC^3$), $v \in W^1_{-\beta}(\Omega,\CC^3)$,
\begin{equation}
\label{eq15f}
b_{\Omega}(u,{v})= \displaystyle{\int_{\Omega} ({\cal A}(\partial_{x_1}, \partial_{x_3}) u) \cdot {{v}} + \int_{\partial \Omega}({\cal B}(\partial_{x_1}, \partial_{x_3}) u) \cdot {{v}}},
\end{equation}
where 
\begin{equation}
\label{eq15g}
({\cal A}(\partial_{x_1}, \partial_{x_3}) u)_i= - \partial_j\sigma_{ij}(u),\, ({\cal B}(\partial_{x_1}, \partial_{x_3}) u)_i= \sigma_{ij}(u) n_j,\, i=1,\ldots,3
\end{equation}
and where in \eqref{eq15f}, \eqref{eq15g} $\cdot$ is the Hermitian product in $\CC^3$, $n_j$, $j=1,\ldots,3$, are the components of the outward normal to $\partial \Omega$.

In order to obtain formula \eqref{eq15h} below, we will apply the method of \cite{Nazarov-Plamenevsky}, p.52. Let $\zeta$ $\in$ ${\cal C}^{\infty}(\RR,\RR)$ be nonnegative and such that $\int_{\RR} \zeta =1$, $\eta \in {\cal C}^{\infty}_0(\RR,\RR)$ be such that $\eta = 1$ in a neighborhood of $0$.
Let $\nu \in \CC$, $\varphi \in H^2(\omega_h,\CC^3)$, $\psi \in H^1(\omega_h,\CC^3)$ and for $\eps >0$, $x \in \RR$, set $u^{\eps}(.,x)= \varphi e^{\nu x} \zeta(x/\eps)/\eps$, $v^{\eps}(.,x) = \psi e^{-{\nu} x} \eta(x)$. It is well-known that $\int_{-\infty}^{+\infty} (\zeta(x/\eps)/\eps) \eta(x) dx$ $\rightarrow $ $\eta(0) =1$ when $\eps \rightarrow 0$. On the other hand by an integration by parts, for $n \in \NN^*$, $\int_{-\infty}^{+\infty} \partial_x^n(\zeta(x/\eps)/\eps) \eta(x) dx$ =  $(-1)^n \int_{-\infty}^{+\infty} (\zeta(x/\eps)/\eps) \partial_x^n\eta(x) dx$ $\rightarrow $ $(-1)^n(\partial_x^n\eta)(0) =0$ when $\eps \rightarrow 0$. Considering $b_{\Omega}(u^{\eps}, \overline{v^{\eps}})$ given by formulas \eqref{eq4c2} and \eqref{eq15f} and taking the limit when $\eps \rightarrow 0$ we get the formula: for all $\varphi \in H^2(\omega_h, \CC^3)$, $\psi \in H^1(\omega_h, \CC^3)$, $\nu \in \CC$,
\begin{eqnarray}
\label{eq15h}
\langle  \mathscr{L}(\nu)\varphi, \psi\rangle &=&
\displaystyle{\int_{\omega^h} ({\cal A}(\partial_{x_1}, \nu) \varphi) \cdot {\overline{\psi}} + \int_{\partial \omega^h}({\cal B}(\partial_{x_1}, \nu) u) \cdot {\overline{\psi}}}.
\end{eqnarray}
This formula shows that the polynomial operator pencil $\mathscr{L}(\nu)$ restricted to $H^2(\omega_h,\CC^3)$ has the same characteristics (eigenvalues, eigenvectors, Jordan chains ...) than the polynomial operator pencil $({\cal A}(\partial_{x_1}, \nu)$, ${\cal B}(\partial_{x_1}, \nu))$ from $H^2(\omega_h,\CC^3)$ into $L^2(\omega_h,\CC^3)$ $\times$ $\RR^2$.
Moreover if $\nu_0$ is an eigenvalue of the polynomial operator pencil $\mathscr{L}(\nu)$ and $\varphi \in H^1(\omega_h,\CC^3)$ is an eigenvector of $\mathscr{L}(\nu)$ corresponding to $\nu_0$, it can be inferred that $\varphi \in H^2(\omega_h,\CC^3)$ (and even $\varphi \in {\cal C}^{\infty}(\overline{\omega_h},\CC^3)$), with the same properties for the corresponding generalized eigenvectors.

The following proposition is fundamental and gives an asymptotics of the solution to \eqref{eq103e}.
\begin{proposition}
\label{pro2sol}
(a) Let $\beta < \gamma$ $\in \RR$.
Let $\nu_1,\ldots,\nu_N$ be the eigenvalues of the pencil $\mathscr{L}(\nu)$ in the strip $-\gamma < Re \nu <-\beta$.
For $i =1,\ldots,N$, let $J_i$, $\kappa_{i,j}, j=1,\ldots,J_i$, $\{\varphi_{j,s}^i\}_{j=1,\ldots,J_i, s=0,\ldots,\kappa_{i,j} -1}$ be the geometric multiplicity of $\nu_i$, the partial multiplicities of $\nu_i$ and a canonical system of Jordan chains of ${\mathscr{L}}(\nu)$ corresponding to $\nu_i$. 
For $i =1,\ldots, N$, let $\{\psi_{j,s}^i\}_{j=1,\ldots,J_i, s=0,\ldots,\kappa_{i,j} -1}$ be the unique canonical system of Jordan chains of ${\mathscr{L}}^*(\nu)$ = ${\mathscr{L}}(-\nu)$ corresponding to $\overline{\nu_i}$ given by Theorem \ref{thkeldysh}.
For $i=1,\ldots,N$, the set of "power-exponential" functions
\begin{equation}
\label{eq15k1}
u^i_{j,s}(x) = e^{\nu_i x} \sum_{\sigma=0}^s \frac{x^{\sigma}}{\sigma !} \varphi_{j,s - \sigma}^i
\end{equation}
for $j=1,\ldots,J_i, s=0,\ldots,\kappa_{i,j} -1$ form a basis of ${\cal N}({\mathscr{L}}(\partial_x), \nu_i)$.
For $i=1,\ldots,N$, the set of "power-exponential" functions
\begin{equation}
\label{eq15k}
w^i_{j,s}(x) = e^{\overline{\nu_i} x} \sum_{\sigma=0}^s \frac{x^{\sigma}}{\sigma !} \psi_{j,s - \sigma}^i
\end{equation}
for $j=1,\ldots,J_i, s=0,\ldots,\kappa_{i,j} -1$ form a basis of ${\cal N}({\mathscr{L}}^*(\partial_x), \overline{\nu_i})$ = ${\cal N}({\mathscr{L}}(-\partial_x), \overline{\nu_i})$ and the set of "power-exponential" functions
\begin{equation}
\label{eq15m}
v^i_{j,s}(x) = w^i_{j,s}(-x) = e^{-\overline{\nu_i} x} \sum_{\sigma=0}^s \frac{(-x)^{\sigma}}{\sigma !} \psi_{j,s - \sigma}^i
\end{equation}
for $j=1,\ldots,J_i, s=0,\ldots,\kappa_{i,j} -1$ form a basis of ${\cal N}({\mathscr{L}}(\partial_x), -\overline{\nu_i})$.
Set
\begin{eqnarray}
\label{eq1az}
&&\Omega_{\beta} = \left \{ x =(x_1,x_3) \in \RR^2, x_1 \in \omega_h, x_3 < 1 \right \}, \nonumber \\
&&\Omega_{\gamma} = \left \{ x =(x_1,x_3) \in \RR^2, x_1 \in \omega_h, -1 < x_3 \right \}
\end{eqnarray}
and let $\varphi_{\beta}$, $\varphi_{\gamma}$ $\in {\cal C}^{\infty}(\overline{\Omega},\RR)$ satisfy the conditions: supp $\varphi_{\beta} \subset \Omega_{\beta}$, supp $\varphi_{\gamma} \subset \Omega_{\gamma}$ and $\varphi_{\beta} + \varphi_{\gamma} \equiv 1$.
Then for  $i=1,\ldots,N$, $j=1,\ldots,J_i, s=0,\ldots,\kappa_{i,j} -1$, $\varphi_{\beta}{v_{j,s}^i} \in {W}^1_{-\beta}(\Omega,\CC^3)$ and $\varphi_{\gamma}{v_{j,s}^i} \in {W}^1_{-\gamma}(\Omega,\CC^3)$.

(b)  Assume moreover that $\mathscr{L}(\nu)$ has no eigenvalue on the lines $Re \nu = - \beta$ and $Re \nu = - \gamma$.
Let $f_{\beta} \in ({W}^1_{-\beta}(\Omega,\CC^3))^*$ and $f_{\gamma} \in ({W}^1_{-\gamma}(\Omega,\CC^3))^*$ be such that $f_{\beta}$ = $f_{\gamma}$ on ${W}^1_{-\beta}(\Omega,\CC^3) \cap {W}^1_{-\gamma}(\Omega,\CC^3)$ and set $u_{\beta} = B_{\Omega,\beta}^{-1} f_{\beta}$, $u_{\gamma} = B_{\Omega,\gamma}^{-1} f_{\gamma}$ (Theorem \ref{thisom}). Then 
\begin{equation}
\label{eq15l}
u_{\beta} =  u_{\gamma} + \sum_{i=1}^N \sum _{j=1}^{J_i} \sum _{s=0}^{\kappa_{i,j} -1} c_{j,s}^i u_{j,s}^i
\end{equation}
a.e. in $\Omega$, 
where 
$\{c_{j,s}^i\}_{i=1,\ldots, N,j=1,\ldots,J_i, s=0,\ldots,\kappa_{i,j} -1}$ are constants (independent of $\varphi_{\beta}$) given by:
\begin{eqnarray}
\label{eq15l1}
c_{j,s}^i &=& \langle f_{\beta}, \varphi_{\beta}\overline{v_{j,\kappa_{i,j}-1 -s}^i}\rangle + \langle f_{\gamma}, \varphi_{\gamma}\overline{v_{j,\kappa_{i,j}-1 -s}^i}\rangle \nonumber \\
&=&
\langle f_{\beta}, \varphi_{\beta}{v_{j,\kappa_{i,j}-1 -s}^i}\rangle_2 +
\langle f_{\gamma}, \varphi_{\gamma}{v_{j,\kappa_{i,j}-1 -s}^i}\rangle_{2}.
\end{eqnarray}
\end{proposition}
\begin{remark}
\label{rem_error}
There is a proof of a similar result in \cite{Kozlov-Mazya-Rossmann}, Theorem 5.4.1, which uses Lemma 5.4.1 of the same reference. But the author of the present paper does not understand the argument p.170 in the proof of this lemma:
"Consequently, the norm of $v_{\rho}$ in $L_2(\Omega\times (-N,N))$ is square integrable with respect to $\rho$ over the interval $(c_1, + \infty)$. Thus this norm tends to zero as $\rho \rightarrow \infty$."
It is well-known that if a function $f: \RR \rightarrow \RR$ is integrable on $\RR$, $f(x)$ does not necessarily tend to zero when $x \rightarrow \pm \infty$. Perhaps it lacks an argument in this proof.
There is another proof of a similar result in \cite{Nazarov-Plamenevsky}, pp.47-48 which uses an argument of density. We apply the same argument of density in the proof of Proposition \ref{pro2sol} where the situation is distinct from that of \cite{Nazarov-Plamenevsky}, pp.47-48 because of the second members which are in a dual space. Moreover the calculus of the coefficients $c_{j,s}^i$ is simpler in the present proof than in the proofs of \cite{Kozlov-Mazya-Rossmann},p.176, \cite{Nazarov-Plamenevsky},pp.50-51 or \cite{Mazya-Plamenevskii}.
\end{remark}
%

\begin{proof}
(a) The basis properties are stated at the end of appendix \ref{oppenc}.
Since for $i=1, \ldots, N$, $\beta < Re(- \overline{\nu_i})< \gamma$, it results that for $i=1,\ldots,N$, $j=1,\ldots,J_i$, $s=0,\ldots,\kappa_{i,j} -1$, $\varphi_{\beta}{v_{j,s}^i} \in {W}^1_{-\beta}(\Omega,\CC^3)$ and $\varphi_{\gamma}{v_{j,s}^i} \in {W}^1_{-\gamma}(\Omega,\CC^3)$.

(b) Given that the embedding from $W^{1}_{-\beta}(\Omega,\CC^3)$ into $W^{0}_{-\beta}(\Omega,\CC^3)$ is continuous and dense, the embedding from $(W^{0}_{-\beta}(\Omega,\CC^3))^*$ into $(W^{1}_{-\beta}(\Omega,\CC^3))^*$ is also continuous and dense (see \cite{Zeidler}, p.265) and thanks to the Riesz representation Theorem and the density of ${\cal C}^{\infty}_0(\overline{{\Omega}},\CC^3)$ in $(W^{0}_{-\beta}(\Omega,\CC^3))^*$ $\cong$ $W^{0}_{\beta}(\Omega,\CC^3)$ it follows that if $f_{\beta} \in ({W}^1_{-\beta}(\Omega,\CC^3))^*$ and $\eps >0$, there exists $h_{\beta} \in {\cal C}^{\infty}_0({\overline{\Omega}},\CC^3)$ satisfying
\begin{equation}
\label{eq15m1}
||f_{\beta} - (.,h_{\beta})_{0,\Omega}||_{-\beta,1,\Omega,*} \leq \eps.
\end{equation}
Likewise if $f_{\gamma} \in ({W}^1_{-\gamma}(\Omega,\CC^3))^*$ and $\eps >0$, there exists $h_{\gamma} \in {\cal C}^{\infty}_0({\overline{\Omega}},\CC^3)$ satisfying 
\begin{equation}
\label{eq15m2}
||f_{\gamma} - (.,h_{\gamma})_{0,\Omega}||_{-\gamma,1,\Omega,*} \leq \eps.
\end{equation}
If $u \in W^{1}_{-\beta}(\Omega,\CC^3)$, it results that $\varphi_{\beta} u \in W^{1}_{-\beta}(\Omega,\CC^3)$ and $\varphi_{\gamma} u \in W^{1}_{-\beta}(\Omega,\CC^3) \cap W^{1}_{-\gamma}(\Omega,\CC^3)$ (because $\beta < \gamma$) and therefore \eqref{eq15m1} implies:
\begin{equation}
\label{eq15n1}
|\langle f_{\beta}, \varphi_{\beta} u \rangle - (u,\varphi_{\beta}h_{\beta})_{0,\Omega}| \leq \eps ||\varphi_{\beta} u||_{-\beta,1,\Omega}
\leq \eps C||u||_{-\beta,1,\Omega}
\end{equation}
and 
\eqref{eq15m2} implies:
\begin{equation}
\label{eq15n2}
|\langle f_{\gamma}, \varphi_{\gamma} u \rangle - (u,\varphi_{\gamma}h_{\gamma})_{0,\Omega}| \leq \eps ||\varphi_{\gamma} u||_{-\gamma,1,\Omega}
\leq \eps C_1||\varphi_{\gamma}u||_{-\beta,1,\Omega}
\leq \eps C_2||u||_{-\beta,1,\Omega},
\end{equation}
where $C>0$, $C_1>0$, $C_2>0$ are constants independent of $u$ and $\eps$. 
Remark that because 
$f_{\beta}$ = $f_{\gamma}$ on ${W}^1_{-\beta}(\Omega,\CC^3) \cap {W}^1_{-\gamma}(\Omega,\CC^3)$,
\begin{equation}
\label{eq15n1a}
\langle f_{\beta}, \varphi_{\beta} u \rangle +
\langle f_{\gamma}, \varphi_{\gamma} u \rangle=\langle f_{\beta}, u \rangle, \, u \in {W}^1_{-\beta}(\Omega,\CC^3).
\end{equation}
Setting $h=\varphi_{\beta}h_{\beta}+\varphi_{\gamma}h_{\gamma} \in {\cal C}^{\infty}_0(\overline{\Omega}, \CC^3)$, adding \eqref{eq15n1} and \eqref{eq15n2} and taking into account \eqref{eq15n1a}, there exists $C_{\beta}>0$ independent of $u$ and $\eps$ satisfying
\begin{equation}
\label{eq15n3}
|\langle f_{\beta},u \rangle - (u,h)_{0,\Omega}| \leq \eps C_{\beta} ||u||_{-\beta,1,\Omega}, \,u \in W^{1}_{-\beta}(\Omega,\CC^3).
\end{equation}
Similarly there exists $C_{\gamma}>0$ independent of $u$ and $\eps$ satisfying

\begin{equation}
\label{eq15n4}
|\langle f_{\gamma},u \rangle - (u,h)_{0,\Omega}| \leq \eps C_{\gamma} ||u||_{-\gamma,1,\Omega},\, u \in W^{1}_{-\gamma}(\Omega,\CC^3).
\end{equation}
%
%
Now let $h \in {\cal C}^{\infty}_0(\overline{\Omega}, \CC^3)$ and set $f=(.,h)_{0,\Omega}$.
For all $\theta \in \RR$,  $f$ defines a continuous map on $W^1_{\theta}(\Omega,\CC^3)$, in particular on $W^1_{-\beta}(\Omega,\CC^3)$ and on $W^1_{-\gamma}(\Omega,\CC^3)$. The partial Fourier-Laplace transform of $f$ in the sense of \eqref{eq106} is holomorphic with respect to $\nu \in \CC$ and satisfies:
\begin{equation}
\label{eq15n5}
\langle \hat{f}(., \nu), \varphi\rangle = \int_{\omega_h} \hat{\overline{h}}(x_1, \nu) \varphi(x_1) d x_1, \, \nu \in \CC,\, \varphi \in H^1(\omega_h,\CC^3).
\end{equation}
As $h \in {\cal C}^{\infty}_0(\overline{\Omega}, \CC^3)$, by integrations by parts, if $n>0$, there exists $C_n >0$ such that if $\nu \in \CC^*$,
$||\hat{\overline{h}}(., \nu)||_{\infty,\omega_h} \leq C_n/{|\nu| ^n}$, consequently there exists $C'_n >0$ such that if $\nu \in \CC^*$,
$||\hat{f}(.,\nu)||_{1,\omega_h,|\nu|,*}\leq C'_n/{|\nu| ^{n+1}}$.
Now by using formula \eqref{eq104a} for $\beta$ and $\gamma$, estimate \eqref{eq103h} and a change of path of integration as in 
\cite{Nazarov-Plamenevsky}, p.47, we get
\begin{eqnarray}
\label{eq15n6}
B_{\Omega,\beta}^{-1}(f\vert_{W^1_{-\beta}(\Omega,\CC^3)})(., x) &=& 
B_{\Omega,\gamma}^{-1}(f\vert_{W^1_{-\gamma}(\Omega,\CC^3)})(., x)\nonumber \\
&+& \sum_{i=1}^n Res_{\nu = \nu_i} 
\{e^{\nu x}(\mathscr{L}(\nu))^{-1}\hat{f}(., \nu)\}.
\end{eqnarray}
In view of \eqref{eq210} we get
\begin{equation}
\label{eq15n7}
Res_{\nu = \nu_i} 
\{e^{\nu x}(\mathscr{L}(\nu))^{-1}\hat{f}(., \nu)\} =
Res_{\nu = \nu_i}
\{e^{\nu x}
\sum _{j=1}^{J_i} \sum _{s=0}^{\kappa_{i,j} -1} \sum_{\sigma=0}^{s} \frac{\langle \hat{f}(., \nu), \overline{\psi_{j,\sigma}^i}\rangle}{(\nu - \nu_i)^{\kappa_{i,j}-s}} \varphi_{j,s-\sigma}^i\}.
\end{equation}
This last term is equal to
\begin{equation}
\label{eq15n8}
e^{\nu_i x}
\sum _{j=1}^{J_i} \sum _{s=0}^{\kappa_{i,j} -1} \sum_{\substack{p\in \NN, q \in \NN,\\ p+q =s}}
\sum_{\substack{p'\in \NN, q' \in \NN,\\p'+q' = \kappa_{i,j}-s-1}}
 \frac{x^{p'}}{(p')!}\frac{\langle \partial_{\nu}^{q'}\hat{f}(., \nu_i), \overline{\psi_{j,p}^i}\rangle}{(q')!} \varphi_{j,q}^i
\end{equation}
written under the form
\begin{equation}
\label{eq15n9}
\sum _{j=1}^{J_i} 
\sum_{k=0}^{\kappa_{i,j} -1}
e^{\nu_i x}
\sum_{\substack{p'\in \NN, q \in \NN,\\ p'+q =k}}
\frac{x^{p'}}{(p')!}
\varphi_{j,q}^i
\sum_{\substack{p\in \NN, q' \in \NN,\\p+q' = \kappa_{i,j}-k-1}}
\frac{\langle \partial_{\nu}^{q'}\hat{f}(., \nu_i), \overline{\psi_{j,p}^i}\rangle}{(q')!}. 
\end{equation}
We can write
\begin{eqnarray}
\label{eq15n10}
{\langle \partial_{\nu}^{q'}\hat{f}(., \nu_i), \overline{\psi_{j,p}^i}\rangle}&=&
\int_{\omega_h} \partial_{\nu}^{q'}\hat{\overline{h}}(x_1, \nu_i) \overline{\psi_{j,p}^i}(x_1) d x_1 \nonumber \\
&=&
\int_{\Omega} \overline{h}(x_1, x_3 ) e^{- \nu_i x_3}(-x_3)^{q'}\overline{\psi_{j,p}^i}(x_1) d x_1 dx_3.
\end{eqnarray}
This shows that the last term in \eqref{eq15n6} is exactly the last term in \eqref{eq15l} (with the notation \eqref{eq15l1}) when $f_{\beta}$=
$f\vert_{W^1_{-\beta}(\Omega,\CC^3)}$,
$f_{\gamma}$=
$f\vert_{W^1_{-\gamma}(\Omega,\CC^3)}$, where $f=(.,h)_{\Omega}$ and $h \in {\cal C}^{\infty}_0(\overline{\Omega},\CC^3)$  (of course in that case $f_{\beta} =f_{\gamma}$ on $W^1_{-\beta}(\Omega,\CC^3) \cap W^1_{-\gamma}(\Omega,\CC^3)$).


Now let $\{h^n_{\beta}\}_{n \in \NN}$ (resp. $\{h^n_{\gamma}\}_{n \in \NN}$) be a sequence in ${\cal C}^{\infty}_0(\overline{\Omega},\CC^3)$ such that $f^n_{\beta} = (., h^n_{\beta})_{0,\Omega} \rightarrow f_{\beta}$ in $({W}^1_{-\beta}(\Omega,\CC^3))^*$ (resp. $f^n_{\gamma} = (., h^n_{\gamma})_{0,\Omega} \rightarrow f_{\gamma}$ in $({W}^1_{-\gamma}(\Omega,\CC^3))^*$) when $n \rightarrow +\infty$ (see \eqref{eq15m1} and \eqref{eq15m2}),
 and for $n \in \NN$ set $h^n= \varphi_{\beta} h_{\beta}^n + \varphi_{\gamma} h_{\gamma}^n$ and $f^n=(.,h^n)_{0,\Omega}$.
Owing to \eqref{eq15n3} and \eqref{eq15n4}, $f^n\vert_{{W}^1_{-\beta}(\Omega,\CC^3)}$  $\rightarrow f_{\beta}$ in $({W}^1_{-\beta}(\Omega,\CC^3))^*$ and  $f^n\vert_{{W}^1_{-\gamma}(\Omega,\CC^3)}$  $\rightarrow f_{\gamma}$ in $({W}^1_{-\gamma}(\Omega,\CC^3))^*$ when $n \rightarrow +\infty$.
From \eqref{eq15n6} applied to $f=f^n$, Theorem \ref{thisom}, the continuity of $c_{j,s}^{i}$ given by \eqref{eq15l1} with respect to $f_{\beta}$ and $f_{\gamma}$, by extracting a subsequence of $\{h_{\beta}^n\}_{n \in \NN}$ and $\{h_{\gamma}^n\}_{n \in \NN}$, we obtain \eqref{eq15l} (a.e in $\Omega$).
Let $\varphi'_{\beta}$, $\varphi'_{\gamma}$ $\in {\cal C}^{\infty}(\overline{\Omega},\RR)$ fulfilling the conditions: supp $\varphi'_{\beta} \subset \Omega_{\beta}$, supp $\varphi'_{\gamma} \subset \Omega_{\gamma}$ and $\varphi'_{\beta} + \varphi'_{\gamma} \equiv 1$. Let $c'^i_{j,s}$ be defined by \eqref{eq15l1} where $\varphi_{\beta}$ and $\varphi_{\gamma}$ are replaced by $\varphi'_{\beta}$ and $\varphi'_{\gamma}$. Since $\varphi_{\beta} - \varphi'_{\beta}$ and $\varphi_{\gamma}-\varphi'_{\gamma}$ have compact supports, since $f_{\beta}$ = $f_{\gamma}$ on ${W}^1_{-\beta}(\Omega,\CC^3) \cap {W}^1_{-\gamma}(\Omega,\CC^3)$ and since $\varphi_{\beta} - \varphi'_{\beta} + \varphi_{\gamma} -\varphi'_{\gamma}\equiv 0$, it results that $c'^i_{j,s}$ = $c^i_{j,s}$.
\end{proof}

\section{Fredholm property in the half-strip $\Omega_+$.}
\setcounter{equation}{0}
\label{Fredholm}
In this section we shall examine the Fredholm properties of the operator associated to the variational formulation in the half-strip.
For $\beta \in \RR$, $k \in \NN$ and $m \in \NN^*$,  define
$W^k_{\beta}(\Omega_+,\CC^m)$ as the completion of ${\cal C} ^{\infty}_0(\overline{\Omega_+},\CC^m)$ for the norm $\lVert u \rVert_{\beta,k,\Omega_+}$ = $\lVert e^{\beta x_3} u \rVert_{k,\Omega_+}$, $H^k_{\Sigma}(\Omega_+,\CC^m)$ = $\{ u \in H^k(\Omega_+,\CC^m), u = 0 \mbox{ on } \Sigma\}$
and
$W^k_{\beta,\Sigma}(\Omega_+,\CC^m)$ = $\{ u \in W^k_{\beta}(\Omega_+,\CC^m), u = 0 \mbox{ on } \Sigma\}$.
The norm on the topological dual  
$(W^k_{\beta,\Sigma}(\Omega_+, \CC^m))^*$ 
of $W^k_{\beta,\Sigma}(\Omega_+, \CC^m)$ will be denoted by 
$||.||_{\beta,\Sigma,k,\Omega_+,*}$.
The sesquilinear form $b_{\Omega_+}$ defined in \eqref{eq2a1} makes sense for  
$u \in W^1_{\beta}(\Omega_+,\CC^3)$ and 
$v \in W^1_{-\beta}(\Omega_+,\CC^3)$ then also for
$u \in W^1_{\beta,\Sigma}(\Omega_+,\CC^3)$ and 
$v \in W^1_{-\beta,\Sigma}(\Omega_+,\CC^3)$ and there exists $C>0$ such that for all $u \in W^1_{\beta}(\Omega_+,\CC^3)$ (resp. $u \in W^1_{\beta,\Sigma}(\Omega_+,\CC^3)$) and 
$v \in W^1_{-\beta}(\Omega_+,\CC^3)$ (resp. $v \in W^1_{-\beta,\Sigma}(\Omega_+,\CC^3)$),

\begin{equation}
\label{eq7a}
|b_{\Omega_+}(u,v)| \leq C||u||_{\beta,1,\Omega_+} ||v||_{-\beta,1,\Omega_+}.
\end{equation}
Define the linear continuous map $B_{\Omega_+,\beta}$: $W^1_{\beta,\Sigma}(\Omega_+,\CC^3)$ $\rightarrow$ $({W}^1_{-\beta,\Sigma}(\Omega_+,\CC^3))^*$ by
\begin{equation}
\label{eq8a}
<B_{\Omega_+,\beta}u,v> = b_{\Omega_+}(u,\overline{v}),\,u \in W^1_{\beta,\Sigma}(\Omega_+,\CC^3),\,v \in W^1_{-\beta,\Sigma}(\Omega_+,\CC^3),
\end{equation}
consequently
\begin{equation}
\label{eq9a}
||B_{\Omega_+,\beta}u||_{-\beta,\Sigma,1,\Omega_+,*} \leq C ||u||_{\beta,1,\Omega_+}, \, u \in W^1_{\beta,\Sigma}(\Omega_+,\CC^3).
\end{equation}
Remark that if $\beta \in \RR$, $(B_{\Omega_+,\beta})^*= B_{\Omega_+,-\beta}$.
Define the sesquilinear Hermitian form $e_{\Omega_+}$ on $H^1(\Omega_+,\CC^3)$ by
\begin{eqnarray}
\label{eq1}
e_{\Omega_+}(u,v)& = &\displaystyle{\int_{\Omega_+} a_{ijkl} \eps_{kl}(u) \eps_{ij}(\overline{v}) + \int_{\Omega_+} u_i\overline{v}_i}\nonumber\\
& =& \displaystyle{\int_{\Omega_+} \lambda \dive u \dive \overline{v} + 2 \mu \eps_{ij}(u) \eps_{ij}(\overline{v}) + \int_{\Omega_+} u_i\overline{v}_i},\,u,v \in H^1(\Omega_+,\CC^3).
\end{eqnarray}
The next two simple lemmas are used in the proof of Proposition \ref{prop1}.
\begin{lemma}
\label{lemma_gd}
Let $\beta \in \RR$, let $\chi \in {\cal C}^{\infty}(\RR,\RR)$ be such that $||\chi||_{\infty}$ $< + \infty$ and $||\chi'||_{\infty}$ $< + \infty$. If $u \in W^{1}_{\beta}(\Omega,\CC^3)$ and $v \in W^{1}_{-\beta}(\Omega,\CC^3)$, then $\chi(x_3) u \in W^{1}_{\beta}(\Omega,\CC^3)$, $\chi(x_3) v \in W^{1}_{-\beta}(\Omega,\CC^3)$ and
\begin{equation}
\label{eq24a}
b_{\Omega}(\chi(x_3) u,v) - b_{\Omega}(u,\chi(x_3) v)= c_{\Omega}(\chi')(u,v)
\end{equation}
with 
\begin{equation}
\label{eq24c}
c_{\Omega}(\chi')(u,v) = \int_{\Omega} D(\chi')(u,v),
\end{equation}
where
\begin{eqnarray}
\label{eq24b}
D(\chi')(u,v)&=& \chi'(x_3) [\lambda (u_3 \dive \overline{v} -  \overline{v}_3 \dive{u})
+2 \mu (u_1 \eps_{13}(\overline{v}) - \overline{v}_1 \eps_{13}(u))\nonumber\\
&+ &2 \mu(u_3 \eps_{33}(\overline{v}) - \overline{v}_3 \eps_{33}(u))+
\mu(u_2 \partial_3 \overline{v}_2- \overline{v}_2 \partial_3 u_2)
].
\end{eqnarray}
Consequently there exists $C>0$ such that if $u \in W^{1}_{\beta}(\Omega,\CC^3)$ and $v \in W^{1}_{-\beta}(\Omega,\CC^3)$,
\begin{equation}
\label{eq24}
|b_{\Omega}(\chi(x_3) u,v) - b_{\Omega}(u,\chi(x_3) v)| \leq 
C ||\chi'||_{\infty} ||u||_{\beta,1, \Omega} ||v||_{-\beta,1, \Omega}.
\end{equation}
Moreover 
\begin{equation}
\label{eq24g}
||\chi(x_3) u||_{\beta,1,\Omega}^2 \leq ||\chi||_{\infty}^2 ||u||_{\beta,1,\Omega}^2 +
||\chi'||_{\infty}^2 ||u||_{\beta,0,\Omega}^2,\,
u \in W^{1}_{\beta}(\Omega,\CC^3)
.
\end{equation}
Likewise let $\chi \in {\cal C}^{\infty}(\RR_+,\RR)$  (where $\RR_+$ is the set of nonnegative real numbers) be such that $||\chi||_{\infty}$ $< + \infty$ and $||\chi'||_{\infty}$ $< + \infty$. If $u \in W^{1}_{\beta}(\Omega_+,\CC^3)$ and $v \in W^{1}_{-\beta}(\Omega_+,\CC^3)$, then $\chi(x_3) u \in W^{1}_{\beta}(\Omega_+,\CC^3)$, $\chi(x_3) v \in W^{1}_{-\beta}(\Omega_+,\CC^3)$ and
\begin{equation}
\label{eq24d}
b_{\Omega_+}(\chi(x_3) u,v) - b_{\Omega_+}(u,\chi(x_3) v)= c_{\Omega_+}(\chi')(u,v)
\end{equation}
with 
\begin{equation}
\label{eq24e}
c_{\Omega_+}(\chi')(u,v) = \int_{\Omega_+} D(\chi')(u,v).
\end{equation}
Consequently there exists $C>0$ such that if $u \in W^{1}_{\beta}(\Omega_+,\CC^3)$ and $v \in W^{1}_{-\beta}(\Omega_+,\CC^3)$,
\begin{equation}
\label{eq24f}
|b_{\Omega_+}(\chi(x_3) u,v) - b_{\Omega_+}(u,\chi(x_3) v)| \leq 
C ||\chi'||_{\infty} ||u||_{\beta,1, \Omega_+} ||v||_{-\beta,1, \Omega_+}
.
\end{equation}
Moreover
\begin{equation}
\label{eq24h}
||\chi(x_3) u||_{\beta,1,\Omega_+}^2 \leq ||\chi||_{\infty}^2 ||u||_{\beta,1,\Omega_+}^2 +
||\chi'||_{\infty}^2 ||u||_{\beta,0,\Omega_+}^2,\,u \in W^{1}_{\beta}(\Omega_+,\CC^3).
\end{equation}
\end{lemma}
%
\begin{proof}
It is a straightforward calculation.
\end{proof}
\begin{lemma}
\label{lemma_c}
Let $\beta \in \RR$. We have: for all $u,v$ $\in H^1(\Omega_+,\CC^3)$,
\begin{equation}
\label{eq10b}
c_{\Omega_+,\beta}(u,v)\overset{def}{=} b_{\Omega_+}(u e^{-\beta x_3}, v e^{\beta x_3})=
\int_{\Omega_+} E(u,v)
\end{equation}
where
\begin{eqnarray}
\label{eq10a}
E(u,v)& = &\displaystyle{ \lambda {\dive} (u e^{-\beta x_3})\dive (\overline{v} e^{\beta x_3}) + 2 \mu \eps_{ij}(u e^{-\beta x_3}) \eps_{ij}(\overline{v} e^{\beta x_3}) - \omega^2 u_i\overline{v}_i}\nonumber\\
&=& \displaystyle{ \lambda (\dive u- \beta u_3)(\dive \overline{v} + \beta \overline{v_3})+ 2 \mu \eps_{11}(u) \eps_{11}(\overline{v})}\nonumber \\
&+&\displaystyle{4 \mu (\eps_{13}(u) -\beta u_1/2 )(\eps_{13}(\overline{v})+ \beta \overline{v_1}/2)}\nonumber\\
&+& \displaystyle{ 2 \mu (\eps_{33}(u) -\beta u_3 )(\eps_{33}(\overline{v})+ \beta \overline{v_3})-\omega^2  u_i\overline{v}_i}.
\end{eqnarray}
Consequently
\begin{eqnarray}
\label{eq10c}
\mbox{Re} c_{\Omega_+,\beta}(u,u) &=& e_{\Omega_+}(u,u) -\beta^2[\mu ||u_1||^2_{0,\Omega_+} + (\lambda+2 \mu) ||u_3||^2_{0,\Omega_+}] \nonumber \\
&-&(\omega^2+1)||u||^2_{0,\Omega_+},\,u \in H^1(\Omega_+,\CC^3).
\end{eqnarray}
\end{lemma}
%
\begin{proof}
It is a staightforward calculation. 
\end{proof}

In order to prove Proposition \ref{prop1} below we shall need a precise form of Korn inequality which a consequence of Theorem 3.1 of \cite{Grabovsky}. Let us recall this theorem.
\begin{theorem}
\label{thGrab}
Let $L>0$ and $0<l<1$. Set $\Omega_{l,L} = (1-l/2,1+l/2) \times (0,L)$. For $u \in H^1(\Omega_{l,L},\RR^2)$ set $e(u)=(\partial_1 u_2 + \partial_2 u_1)/2$. Set $E_{l,L} = \{u \in H^1(\Omega_{l,L},\RR^2), u_2 \vert_{x_2=0} = 0 \}$. If $L>0$, there exists $C(L) >0$ such that if $0<l<1$ and $u \in E_{l,L}$,
\begin{equation}
\label{eq10c1}
 |u|^2_{1,\Omega_{l,L}} \leq C(L) ||e(u)||_{0,\Omega_{l,L}}(||u||_{0,\Omega_{l,L}}/l + ||e(u)||_{0,\Omega_{l,L}}).
\end{equation}
\end{theorem}
A consequence of Theorem \ref{thGrab} is the following one
\begin{theorem}
\label{thKorn}
For $R>0$ set $\Omega_{R}= (-h,h) \times (0,R)$. 
There exists $C>0$ such that for all $R>0$, for all $u \in H^1(\Omega_{R},\CC^3)$ satisfying $u\vert_{\Sigma}=0$,
\begin{equation}
\label{eq10c2}
 |u|^2_{1,\Omega_{R}} \leq C \left(\sum_{i=1}^3 \sum_{j=1}^3 ||\eps_{ij}(u)||^2_{0,\Omega_{R}}+ ||u||^2_{0,\Omega_{R}}\right).
\end{equation}
\end{theorem}
%
\begin{proof}
If $R>0$ and $u \in H^1(\Omega_{R},\CC^2)$ satisfies $u\vert_{\Sigma}=0$, define $v \in H^1(\Omega_{2h/R,1},\CC^2)$ by $v(x_1,x_3) = u(R(x_1-1),Rx_3)$ for $(x_1,x_3) \in \Omega_{2h/R,1}$. Apply Theorem \ref{thGrab} with $l = 2h/R$ and $L=1$. We have $||v||_{0,\Omega_{2h/R,1}} = ||u||_{0,\Omega_{R}}/R$, $|v|_{1,\Omega_{2h/R,1}} = |u|_{1,\Omega_{R}}$, $||e(v)||_{0,\Omega_{2h/R,1}} = ||e(u)||_{0,\Omega_{R}}$. We obtain:
\begin{equation}
\label{eq10c3}
 |u|^2_{1,\Omega_{R}} \leq C(1) ||e(u)||_{0,\Omega_{R}}(||u||_{0,\Omega_{R}}/{2h} + ||e(u)||_{0,\Omega_{R}}).
\end{equation}
\eqref{eq10c2} is a consequence of \eqref{eq10c3}.
\end{proof}

The following proposition is fundamental to show that if $\beta$ is such that $\mathscr{L}(\nu)$ has no eigenvalue on the lines $Re \nu = - \beta$ and $Re \nu = \beta$, $B_{\Omega_+,\beta}$ is a Fredholm operator (Proposition \ref{propfredholm}).
\begin{proposition}
\label{prop1}
Let $\beta \in \RR$ be such that $\mathscr{L}(\nu)$ has no eigenvalue on the line $Re \nu = - \beta$.
There exists $C>0$ and $R>0$ such that if $\Omega_R = (-h,h) \times (0,R)$,
\begin{equation}
\label{eq25}
||u||_{\beta,1,\Omega_+} \leq C(||B_{\Omega_+,\beta}u||_{-\beta,\Sigma,1,\Omega_+,*} + ||u||_{0,\Omega_R}),\, u \in W^1_{\beta,\Sigma}(\Omega_+,\CC^3).
\end{equation}
\end{proposition}
%
\begin{proof}
Let $\chi \in {\cal C}^{\infty}_0(\RR_+,\RR)$ with the properties that $\chi =1$ on $[0,1]$, supp $\chi$ $\subset$ $[0,2]$, $0 \leq \chi \leq 1$ and for $n \in \NN^*$, set $\chi_n(.)= \chi(./n)$.
We can write
\begin{equation}
\label{eq21}
||u||_{\beta,1,\Omega_+} \leq ||\chi_n(x_3) u||_{\beta,1,\Omega_+} + ||(1-\chi_n(x_3)) u||_{\beta,1,\Omega_+}, u \in W^1_{\beta}(\Omega_+,\CC^3), n \in \NN^* .
\end{equation}
Thanks to Theorem \ref{thKorn} and with notation \eqref{eq1}, there exists $C>0$ satisfying
\begin{equation}
\label{eq30b}
||\chi_n(x_3) u||^2_{1,\Omega_+} \leq C e_{\Omega_+}(\chi_n(x_3) u, \chi_n(x_3) u),\,u \in H^1_{\Sigma}(\Omega_+,\CC^3),n \in \NN^*. 
\end{equation}
It follows that for all $u \in W^1_{\beta,\Sigma}(\Omega_+,\CC^3)$, $n \in \NN^*$, 
\begin{equation}
\label{eq30c}
||\chi_n(x_3) e^{\beta x_3}u||^2_{1,\Omega_+} \leq C e_{\Omega_+}(\chi_n(x_3) e^{\beta x_3}u, \chi_n(x_3) e^{\beta x_3}u). 
\end{equation}
From \eqref{eq10c}, there exists $C>0$ satisfying
\begin{equation}
\label{eq31a}
e_{\Omega_+}(u, u) \leq \mbox { Re} c_{\Omega_+,\beta}(u,u) + C ||u||^2_{0,\Omega_+},\,u \in H^1(\Omega_+,\CC^3).
\end{equation}
Taking into account \eqref{eq30c} and \eqref{eq31a}, there exists $C>0$ such that if $u$ $\in W^1_{\beta,\Sigma}(\Omega_+,\CC^3)$ and $n \in \NN^*$,
\begin{eqnarray}
\label{eq32a}
||\chi_n(x_3) u||^2_{\beta,1,\Omega_+} &
\leq &
 C (\mbox { Re} c_{\Omega_+,\beta}(\chi_n(x_3) e^{\beta x_3}u, \chi_n(x_3) e^{\beta x_3}u) +  ||\chi_n(x_3) u||^2_{\beta,0,\Omega_+})\nonumber\\
 & \leq &
 C  (|c_{\Omega_+,\beta}(\chi_n(x_3) e^{\beta x_3}u, \chi_n(x_3) e^{\beta x_3}u)| +  ||\chi_n(x_3) u||^2_{\beta,0,\Omega_+})\nonumber\\
  & = &
 C  (|b_{\Omega_+}(\chi_n(x_3) u, \chi_n(x_3) e^{2\beta x_3}u)| +  ||\chi_n(x_3) u||^2_{\beta,0,\Omega_+})\nonumber\\
&\leq & C (||B_{\Omega_+,\beta}(\chi_n(x_3) u)||_{-\beta,\Sigma,1,\Omega_+,*}
||\chi_n(x_3) u||_{\beta,1,\Omega_+} \nonumber \\
&+& ||\chi_n(x_3) u||^2_{\beta,0,\Omega_+}),
\end{eqnarray}
accordingly there exists $C>0$ such that if $u$ $\in W^1_{\beta,\Sigma}(\Omega_+,\CC^3)$ and $n \in \NN^*$,
\begin{equation}
\label{eq33a}
||\chi_n(x_3) u||_{\beta,1,\Omega_+} \leq C (||B_{\Omega_+,\beta}(\chi_n(x_3) u)||_{-\beta,\Sigma,1,\Omega_+,*} + ||\chi_n(x_3) u||_{\beta,0,\Omega_+}).
\end{equation}
On the other hand due to \eqref{eq24f} there exists $C>0$ such that if $u \in W^1_{\beta,\Sigma}(\Omega_+,\CC^3)$ and $n \in \NN^*$,
\begin{eqnarray}
\label{eq34}
||B_{\Omega_+,\beta}(\chi_n(x_3) u)||_{-\beta,\Sigma,1,\Omega_+,*}&=& 
\underset{v \in W^1_{-\beta,\Sigma}(\Omega_+,\CC^3), ||v||_{-\beta,1,\Omega_+} \leq 1}{\mbox{sup}} |b_{\Omega_+}(\chi_n(x_3) u, v)|\nonumber\\
&\leq& \underset{v \in W^1_{-\beta,\Sigma}(\Omega_+,\CC^3), ||v||_{-\beta,1,\Omega_+} \leq 1}{\mbox{sup}} |b_{\Omega_+}(u, \chi_n(x_3) v)| \nonumber\\
&+& C||\chi'/n||_{\infty} ||u||_{\beta,1, \Omega_+}
\end{eqnarray}
and from \eqref{eq24h}, there exists $C>0$ such that if $u \in W^1_{\beta,\Sigma}(\Omega_+,\CC^3)$, $v \in W^1_{-\beta,\Sigma}(\Omega_+,\CC^3)$ and $n \in \NN^*$,
\begin{eqnarray}
\label{eq35b}
|b_{\Omega_+}(u, \chi_n(x_3) v)| &\leq &
||B_{\Omega_+,\beta}(u)||_{-\beta,\Sigma,1,\Omega_+,*}
||\chi_n(x_3) v||_{-\beta,1,\Omega_+}\nonumber\\
&\leq&
C||B_{\Omega_+,\beta}(u)||_{-\beta,\Sigma,1,\Omega_+,*}
||v||_{-\beta,1,\Omega_+}.
\end{eqnarray}
Owing to \eqref{eq33a}-\eqref{eq35b} there exists $C>0$ such that if $u \in W^1_{\beta,\Sigma}(\Omega_+,\CC^3)$ and $n \in \NN^*$,
\begin{eqnarray}
\label{eq37}
||\chi_n(x_3) u||_{\beta,1,\Omega_+} &\leq& C (||B_{\Omega_+,\beta}(u)||_{-\beta,\Sigma,1,\Omega_+,*} +  ||\chi_n(x_3) u||_{\beta,0,\Omega_+} \nonumber \\
&+& ||\chi'/n||_{\infty} ||u||_{\beta,1,\Omega_+}).
\end{eqnarray}
Moreover $\beta$ is such that $\mathscr{L}(\nu)$ has no eigenvalue on the line $Re \nu = - \beta$, therefore
 $B_{\Omega,\beta}$ is an isomorphism from 
$W^1_{\beta}(\Omega,\CC^3)$ onto $({W}^1_{-\beta}(\Omega,\CC^3))^*$ (Theorem \ref{thisom}): there exists $C>0$ such that if $u \in W^1_{\beta,\Sigma}(\Omega_+,\CC^3)$ and $n \in \NN^*$,
\begin{eqnarray}
\label{eq38}
||(1-\chi_n(x_3)) u||_{\beta,1,\Omega_+} &=& ||(1-\chi_n(x_3)) u||_{\beta,1,\Omega}\nonumber \\
&\leq& C ||B_{\Omega,\beta}((1-\chi_n(x_3))u)||_{-\beta,1,\Omega,*}.
\end{eqnarray}
On the other hand from \eqref{eq24f} there exists $C>0$ such that if $u \in W^1_{\beta,\Sigma}(\Omega_+,\CC^3)$ and $n \in \NN^*$,
\begin{eqnarray}
\label{eq39}
&&||B_{\Omega,\beta}((1-\chi_n(x_3)) u)||_{-\beta,1,\Omega,*}=
\underset{v \in W^1_{-\beta}(\Omega,\CC^3), ||v||_{-\beta,1,\Omega} \leq 1}{\mbox{sup}} |b_{\Omega}((1-\chi_n(x_3)) u, v)|\nonumber\\
&&= 
\underset{v \in W^1_{-\beta}(\Omega,\CC^3), ||v||_{-\beta,1,\Omega} \leq 1}{\mbox{sup}} |b_{\Omega_+}((1-\chi_n(x_3)) u, (1-\chi_{n/2}(x_3))v)|\nonumber\\
&&\leq
\underset{v \in W^1_{-\beta}(\Omega,\CC^3), ||v||_{-\beta,1,\Omega} \leq 1}{\mbox{sup}} |b_{\Omega_+}(u,(1-\chi_n(x_3)) v)|
+ C||\chi'/n||_{\infty} ||u||_{\beta,1, \Omega_+}
\end{eqnarray}
and thanks to \eqref{eq24g} there exists $C>0$ such that if $u \in W^1_{\beta,\Sigma}(\Omega_+,\CC^3)$, $v \in W^1_{-\beta}(\Omega,\CC^3)$ and $n \in \NN^*$,
\begin{eqnarray}
\label{eq35c}
|b_{\Omega_+}(u, (1-\chi_n(x_3) )v)| &\leq &
||B_{\Omega_+,\beta}(u)||_{-\beta,\Sigma,1,\Omega_+,*}
||(1-\chi_n(x_3)) v||_{-\beta,1,\Omega_+}\nonumber\\
&\leq&
C||B_{\Omega_+,\beta}(u)||_{-\beta,\Sigma,1,\Omega_+,*}
||v||_{-\beta,1,\Omega}.
\end{eqnarray}
In view of  \eqref{eq38}-\eqref{eq35c}, there exists $C>0$ such that if $u \in W^1_{\beta,\Sigma}(\Omega_+,\CC^3)$ and $n \in \NN^*$,
\begin{equation}
\label{eq42}
||(1- \chi_n(x_3))u||_{\beta,1,\Omega_+} \leq C (||B_{\Omega_+,\beta}(u)||_{-\beta,\Sigma,1,\Omega_+,*}  +  ||\chi'/n||_{\infty} ||u||_{\beta,1,\Omega_+}).
\end{equation}
Thanks to \eqref{eq21}, \eqref{eq37}, \eqref{eq42}, there exists $C>0$ such that if $u \in W^1_{\beta,\Sigma}(\Omega_+,\CC^3)$ and $n \in \NN^*$,
\begin{equation}
\label{eq43}
||u||_{\beta,1,\Omega_+} \leq C (||B_{\Omega_+,\beta}(u)||_{-\beta,\Sigma,1,\Omega_+,*}  +  ||\chi_n(x_3) u||_{\beta,0,\Omega_+} +  ||\chi'/n||_{\infty} ||u||_{\beta,1,\Omega_+}).
\end{equation}
Choose $n$ with the property that $C ||\chi'/n||_{\infty} \leq 1/2$ and the proposition follows.
\end{proof}

From \cite{Bourgeois-Chesnel-Fliss}, Lemma A.1, it follows that if 
$\beta \in \RR$ is such that $\mathscr{L}(\nu)$ has no eigenvalue on the line $Re \nu = - \beta$, the kernel of $B_{\Omega_+,\beta}$ is finite-dimensional and its range is closed.
On the other hand we know that for all $\beta \in \RR$, $(B_{\Omega_+,\beta})^*$ = $B_{\Omega_+,-\beta}$ so that if $\beta \in \RR$ is such that $\mathscr{L}(\nu)$ has no eigenvalue on the lines $Re \nu = - \beta$ and $Re \nu = \beta$, moreover the kernel of $B_{\Omega_+,-\beta}$ = $(B_{\Omega_+,\beta})^*$ is finite-dimensional and consequently the range of $B_{\Omega_+,\beta}$ (which is closed) has a finite codimension (see \cite{Kato}, Lemma 1.40, p.141). Since in the previous reasoning $\beta$ can be replaced by $-\beta$ we have shown
\begin{proposition}
\label{propfredholm}
If $\beta \in \RR$ is such that $\mathscr{L}(\nu)$ has no eigenvalue on the lines $Re \nu = - \beta$ and $Re \nu = \beta$,
the operators $B_{\Omega_+,\beta}$  and $B_{\Omega_+,-\beta}$ are Fredholm.
\end{proposition}
\section{Asymptotics in the half-strip $\Omega_+$.}
\label{Asymptotics}
\setcounter{equation}{0}
This section is devoted to the asymptotics of solutions to the variational formulation in the half-strip (Proposition \ref{asymt}).
Hereafter in this paper, $\chi$ is a function such that $\chi \in {\cal C}^{\infty}(\RR,\RR)$ and $\chi(x) =0$ for $x \leq 1$, $\chi(x) =1$ for $x \geq 2$.

In this section  $\beta$  and  $\gamma$ are real numbers such that $\beta < \gamma$ and we shall use the notations of Proposition \ref{pro2sol}, (a).
The eigenvalues $\nu_i$ of ${\mathscr{L}}(\nu)$ in Proposition \ref{pro2sol}, (a) fulfill the condition $-\gamma < Re \nu_i < - \beta$ therefore with the notation \eqref{eq15k1} of this proposition, for $m \in \NN$, $e^{\beta x_3}\chi(x_3) u^i_{j,s}$ $\in$ $H^{m}(\Omega_+,\CC^3)$ (recall that for $m \in \NN$, $u^i_{j,s} \in H^m(\omega_h,\CC^3)$), consequently $\chi(x_3) u^i_{j,s}$ $\in$ $W^{m}_{\beta,\Sigma}(\Omega_+,\CC^3)$. 
Let $\kappa$ be the sum of the algebraic multiplicities of the eigenvalues of $\mathscr{L}(\nu)$ (strictly) between the lines $Re \nu = - \beta$ and $Re \nu = - \gamma$.
With the notation \eqref{eq15k1} of Proposition \ref{pro2sol}, (a), denote by $\{W_k\}_{k=1,\ldots,\kappa}$ the family $\{\chi(x_3) u_{j,s}^i\}_{i=1,\ldots, N,j=1,\ldots,J_i, s=0,\ldots,\kappa_{i,j} -1}$ (these functions will be named "waves").
Let ${\cal D}_W$ be the vector space spanned by these $\kappa$ functions. Lemma \ref{indep} implies that the functions $W_k, k =1,\ldots,\kappa$ are linearly independent modulo $W^{1}_{\gamma}(\Omega_+,\CC^3)$ and linearly independent, then the sum ${\cal D}_W$ $+$ $W^{1}_{\gamma}(\Omega_+,\CC^3)$ is direct and ${\cal D}_W$ is $\kappa$-dimensional.
Consider the space ${\cal D}$ = ${\cal D}_W$ $\oplus$ $W^{1}_{\gamma}(\Omega_+,\CC^3)$ ($W^{1}_{\gamma}(\Omega_+,\CC^3)$ $\subset$ ${\cal D}$ $\subset$ $W^{1}_{\beta}(\Omega_+,\CC^3)$). 
The space ${\cal D}$ equipped with the norm: for $u= u_W + u_{\gamma}$ $\in {\cal D}$ with $u_W = \sum_{k=1}^{\kappa} a_k W_k$, $a_1,\ldots,a_{\kappa} \in \CC$ and $u_{\gamma} \in W^{1}_{\gamma}(\Omega_+,\CC^3)$,  
\begin{equation}
\label{eq15n12c2}
\displaystyle{||u||^2_{{\cal D}} = \sum_{k=1}^{\kappa} |a_k|^2 + ||u_{\gamma}||^2_{\gamma,1,\Omega_+}}
\end{equation}
is a Hilbert space for the inner product associated with this norm.
%
In the same way the sum ${\cal D}_W$ $+$ $W^{1}_{\gamma,\Sigma}(\Omega_+,\CC^3)$ is direct and the space ${\cal D}_{\Sigma}$ =   ${\cal D}_W$ $\oplus$ $W^{1}_{\gamma,\Sigma}(\Omega_+,\CC^3)$ ($W^{1}_{\gamma,\Sigma}(\Omega_+,\CC^3)$ $\subset$ ${\cal D}_{\Sigma}$ $\subset$ $W^{1}_{\beta,\Sigma}(\Omega_+,\CC^3)$) is a closed subspace of ${\cal D}$.

Since $\beta < \gamma$,  ${W}^1_{\gamma,\Sigma}(\Omega_+,\CC^3)$ $\subset$ ${W}^1_{\beta,\Sigma}(\Omega_+,\CC^3)$ and ${W}^1_{-\beta,\Sigma}(\Omega_+,\CC^3)$ $\subset$ ${W}^1_{-\gamma,\Sigma}(\Omega_+,\CC^3)$. As ${W}^1_{-\beta,\Sigma}(\Omega_+,\CC^3)$ is dense in ${W}^1_{-\gamma,\Sigma}(\Omega_+,\CC^3)$, the adjoint of the inclusion map between these two spaces,  which is the restriction map
\begin{equation}
\label{eq300}
i_{\beta,\gamma}: u \in ({W}^1_{-\gamma,\Sigma}(\Omega_+,\CC^3))^* \mapsto u\vert_{{W}^1_{-\beta,\Sigma}(\Omega_+,\CC^3)} \in ({W}^1_{-\beta,\Sigma}(\Omega_+,\CC^3))^*
\end{equation}
is an embedding (one-to-one, continuous and with a dense range), see \cite{Zeidler}, p.265.
The following proposition, which is fundamental, is a consequence of Proposition \ref{pro2sol} and will be applied several times in the sequel.
\begin{proposition}
\label{asymt}
Assume that $\mathscr{L}(\nu)$ has no eigenvalue on the lines $Re \nu = - \beta$ and $Re \nu = - \gamma$ and let $f \in$ $({W}^1_{-\gamma,\Sigma}(\Omega_+,\CC^3))^*$.
Assume that $u_{\beta} \in {W}^1_{\beta,\Sigma}(\Omega_+,\CC^3)$ satisfies $B_{\Omega_+,\beta} u_{\beta} =i_{\beta,\gamma}f$ = $f\vert_{{W}^1_{-\beta,\Sigma}(\Omega_+,\CC^3)}$.
In that case there exists a family of complex numbers $\{c_k\}_{k=1,\ldots,\kappa}$ satisfying
\begin{equation}
\label{eq15l2}
u_{\beta} =  \sum_{k=1}^{\kappa} c_k W_k + u_{\gamma}
\end{equation}
where $u_{\gamma} \in {W}^1_{\gamma,\Sigma}(\Omega_+,\CC^3)$, in other words $u_{\beta} \in {\cal D}_{\Sigma}$.
\end{proposition}
%
\begin{proof}
The relation $B_{\Omega_+,\beta} u_{\beta} = f\vert_{{W}^1_{-\beta,\Sigma}(\Omega_+,\CC^3)}$ means
\begin{equation}
\label{eq15l3}
b_{\Omega_+}(u_{\beta},\overline{v}) = \langle f, v \rangle,\,v \in W^1_{-\beta,\Sigma}(\Omega_+,\CC^3).
\end{equation}
It follows that $\chi(x_3) u_{\beta} \in {W}^1_{\beta}(\Omega,\CC^3)$ and due to \eqref{eq24d}, that if $v \in {W}^1_{-\beta,\Sigma}(\Omega_+,\CC^3)$,
\begin{equation}
\label{eq45c1}
b_{\Omega_+}(\chi(x_3) u_{\beta}, \overline{v}) = b_{\Omega_+}(u_{\beta}, \chi(x_3) \overline{v}) + c_{\Omega_+}(\chi')(u_{\beta}, \overline{v}) = \langle f, \chi(x_3) v\rangle + c_{\Omega_+}(\chi')(u_{\beta}, \overline{v}).
\end{equation}
Let $\chi_1 \in {\cal C}^{\infty}(\RR,\RR)$ be such that $\chi_1(x) =0$ for $x \leq 1/2$, $\chi_1(x) =1$ for $x \geq 1$ so that $\chi_1 \chi \equiv \chi$. Now if $v \in {W}^1_{-\beta}(\Omega,\CC^3)$ then $\chi_1(x_3) v \in {W}^1_{-\beta,\Sigma}(\Omega,\CC^3)$, $b_{\Omega_+}(\chi(x_3) u_{\beta}, \overline{\chi_1(x_3) v})$ = $b_{\Omega}(\chi(x_3) u_{\beta}, \overline{v})$, $c_{\Omega_+}(\chi')(u_{\beta}, \overline{\chi_1(x_3) v})$ = $c_{\Omega_+}(\chi')(u_{\beta}, \overline{v})$ so that \eqref{eq45c1} yields
\begin{equation}
\label{eq45d1}
b_{\Omega}(\chi(x_3) u_{\beta}, \overline{v}) = \langle f, \chi(x_3) v\rangle + c_{\Omega_+}(\chi')(u_{\beta}, \overline{v}) , \, v \in {W}^1_{-\beta}(\Omega
,\CC^3).
\end{equation}
The second member in \eqref{eq45d1} is continuous with respect to $v \in {W}^1_{-\beta}(\Omega,\CC^3)$ therefore if $v \in  {W}^1_{-\beta}(\Omega,\CC^3)$ it can be written under the form $\langle f_{\beta}, v \rangle$ where $f_{\beta} \in ({W}^1_{-\beta}(\Omega,\CC^3))^*$. Moreover it is continuous with respect to $v \in {W}^1_{-\gamma}(\Omega,\CC^3)$ consequently if $v \in  {W}^1_{-\gamma}(\Omega,\CC^3)$ it can be written under the form $\langle f_{\gamma}, v \rangle$ where $f_{\gamma} \in ({W}^1_{-\gamma}(\Omega,\CC^3))^*$ and $f_{\beta} = f_{\gamma}$ on ${W}^1_{-\beta}(\Omega,\CC^3) \cap {W}^1_{-\gamma}(\Omega,\CC^3)$.
Eq. \eqref{eq45d1} means $B_{\Omega,\beta} (\chi(x_3) u_{\beta}) = f_{\beta}$. By writing $u_{\beta}$ under the form $u_{\beta} =  \chi(x_3) (\chi(x_3) u_{\beta}) + (1- \chi^2(x_3))u_{\beta}$ and applying  Proposition \ref{pro2sol}, (b) we get \eqref{eq15l2}.
\end{proof}

The next simple lemma will be used several times in the sequel to extend some variational formulations to a larger space.
\begin{lemma}
\label{noyau1}
%
(a) If $u \in W^{1}_{\gamma,\Sigma}(\Omega_+,\CC^3)$, then $B_{\Omega_+,\beta} u = i_{\beta,\gamma}B_{\Omega_+,\gamma} u$ = 
$B_{\Omega_+,\gamma} u\vert_{W^{1}_{-\beta,\Sigma}(\Omega_+,\CC^3)}$.

(b) Assume that $f \in (W^{1}_{-\gamma,\Sigma}(\Omega_+,\CC^3))^*$ and $u \in W^{1}_{\gamma,\Sigma}(\Omega_+,\CC^3)$ satisfy $B_{\Omega_+,\beta} u = i_{\beta,\gamma}f$ = $f\vert_{W^{1}_{-\beta,\Sigma}(\Omega_+,\CC^3)}$, that is to say
\begin{equation}
\label{eq1000a}
b_{\Omega_+}(u,v) = \langle f,v\rangle, \, v \in W^{1}_{-\beta,\Sigma}(\Omega_+,\CC^3).
\end{equation}
Then
\begin{equation}
\label{eq1001a}
b_{\Omega_+}(u,v) = \langle f,v\rangle, \, v \in W^{1}_{-\gamma,\Sigma}(\Omega_+,\CC^3),
\end{equation}
namely $B_{\Omega_+,\gamma} u = f$.
In particular $W^{1}_{\gamma,\Sigma}(\Omega_+,\CC^3) \cap KerB_{\Omega_+, \beta}$  = $KerB_{\Omega_+, \gamma}$.
\end{lemma}
%
\begin{proof}
Part (a) is trivial and part (b) is simply a consequence of the density of $W^{1}_{-\beta,\Sigma}(\Omega_+,\CC^3)$ in $W^{1}_{-\gamma,\Sigma}(\Omega_+,\CC^3)$.
%
\end{proof}

As in \eqref{eq15f}, \eqref{eq15g}, thanks to Green formula for an open set with Lipschitz boundary, it results that for $u \in W^2_{\beta}(\Omega_+,\CC^3)$ and $v \in W^1_{-\beta}(\Omega_+,\CC^3)$,
\begin{equation}
\label{eq15f1a}
b_{\Omega_+}(u,{v})= {b}^1_{\Omega_+}(u,{v})
\end{equation}
where 
\begin{equation}
\label{eq15f1b}
{b}^1_{\Omega_+}(u,{v})= \displaystyle{\int_{\Omega_+} ({\cal A}(\partial_{x_1}, \partial_{x_3}) u) \cdot {{v}} + \int_{\partial \Omega_+}({\cal B}(\partial_{x_1}, \partial_{x_3}) u) \cdot {{v}}},
\end{equation}
\begin{equation}
\label{eq15g1a}
({\cal A}(\partial_{x_1}, \partial_{x_3}) u)_i= - \partial_j\sigma_{ij}(u),\, ({\cal B}(\partial_{x_1}, \partial_{x_3}) u)_i= \sigma_{ij}(u) n_j,\, i=1,\ldots,3
\end{equation}
and where in \eqref{eq15f1b}, \eqref{eq15g1a}, $\cdot$ is the Hermitian product in $\CC^3$, $n_j$, $j=1,\ldots,3$, are the components of the outward normal to $\partial \Omega_+$.

The next lemma shows that the form ${b}^1_{\Omega_+}$ can be extended to other spaces.
\begin{lemma}
\label{lemext}
If $\theta \in \RR$, one can extend ${b}^1_{\Omega_+}$ (still denoted ${b}^1_{\Omega_+}$) on ${\cal D}_W \times W^1_{\theta}(\Omega_+,\CC^3)$  by formula \eqref{eq15f1b} so that for all $k=1,\ldots,\kappa$, there exists $C_{k,\theta} >0$ satisfying
\begin{equation}
\label{eq15f1a2}
|{b}^1_{\Omega_+}(W_k,v)| \leq C_{k,\theta} ||v||_{\theta,1,\Omega_+},\,v \in W^1_{\theta}(\Omega_+,\CC^3).
\end{equation}
%
%
\end{lemma}
%
\begin{proof}
%
%
%
Formula \eqref{eq15f1a} can be applied to $u$ = $\chi(x_3)u^i_{j,s}$ $\in$  $W^2_{\beta,\Sigma}(\Omega_+,\CC^3)$. On the other hand
${\cal A}(\partial_{x_1}, \partial_{x_3}) (\chi(x_3)u^i_{j,s})= \chi(x_3) {\cal A}(\partial_{x_1}, \partial_{x_3}) u^i_{j,s}$
+ terms in $\partial_x \chi$, 
${\cal A}(\partial_{x_1}, \partial_{x_3}) u^i_{j,s}\equiv 0$, 
with similar properties for ${\cal B}(\partial_{x_1}, \partial_{x_3}) (\chi(x_3)u^i_{j,s})$
and $\partial_x \chi$ has a compact support. 
Consequently for $u$ = $\chi(x_3)u^i_{j,s}$ $\in$  $W^2_{\beta,\Sigma}(\Omega_+,\CC^3)$, if $\theta \in \RR$ and $v \in W^1_{\theta}(\Omega_+,\CC^3)$, the second member of \eqref{eq15f1b} makes sense and 
we shall define ${b}^1_{\Omega_+}(\chi(x_3)u^i_{j,s}, v)$ by formula \eqref{eq15f1b}. We obtain \eqref{eq15f1a2} and we can extend ${b}^1_{\Omega_+}$ on ${\cal D}_W \times W^1_{\theta}(\Omega_+,\CC^3)$ by linearity.
\end{proof}
\section{Radiation conditions in the half-strip $\Omega_+$.}
\label{Radiation}
\setcounter{equation}{0}
In this section we introduce the complex symplectic form $q_{\Omega_+}$ (Eq. \eqref{eq15n11}). Thanks to this form it is possible to define incoming and outgoing "waves" according to the Mandelstam radiation principle.
In the sequel we shall make the following assumption
\begin{assumption}
\label{ass1}
$\delta >0$ is chosen such that in the strip $-\delta < \mbox{Re} \nu < \delta$, the eigenvalues of the operator pencil ${\mathscr{L}}(\nu)$ are only on the imaginary axis $\mbox{Re} \nu  =0$ and $\mathscr{L}(\nu)$ has no eigenvalue on the lines $Re \nu = - \delta$ and $Re \nu =  \delta$. 
\end{assumption}

Let $\nu_1 = i \omega_1$, $\ldots$, $\nu_{N} = i \omega_N$ be these eigenvalues where $\omega_j \in \RR, j =1,\ldots,N$.
Given that $B_{\Omega_+,\delta}$ and $B_{\Omega_+,-\delta}$ are Fredholm operators (Proposition \ref{propfredholm}), $KerB_{\Omega_+,-\delta}$ and $KerB_{\Omega_+,\delta}$ $\subset$ $KerB_{\Omega_+,-\delta}$ are finite-dimensional. Let $T$ be the codimension of  $KerB_{\Omega_+,\delta}$ in $KerB_{\Omega_+,-\delta}$.
Let $\kappa$ be the sum of the algebraic multiplicities of the imaginary eigenvalues $i \omega_j$, $j=1,\ldots,N$ of the operator pencil ${\mathscr{L}}(\nu)$.

Owing to Proposition \ref{ker*}
\begin{equation}
\label{eq43a}
dim (KerB_{\Omega_+,-\delta}/KerB_{\Omega_+,\delta})+dim (KerB_{\Omega_+,\delta}^*/KerB_{\Omega_+,-\delta}^*) = \kappa.
\end{equation}
Given that $B_{\Omega_+,\delta}^*$ = $B_{\Omega_+,-\delta}$ and $B_{\Omega_+,-\delta}^*$ = $B_{\Omega_+,\delta}$ it follows that $\kappa=2T$.
The results of section \ref{Asymptotics} can be applied with $\beta = -\delta$, $\gamma = \delta$ and $\kappa = 2T$.

In that case ${\cal D}_W$ is the vector space spanned by the $2T$ functions $W_1,\ldots, W_{2T}$. The functions $W_k, k =1,\ldots,2T$ are linearly independent modulo $W^{1}_{\delta}(\Omega_+,\CC^3)$ and linearly independent, the sum $W^{1}_{\delta}(\Omega_+,\CC^3)$ $+$ ${\cal D}_W$ is direct and ${\cal D}_W$ is $2T$-dimensional.
The space ${\cal D}$ = $W^{1}_{\delta}(\Omega_+,\CC^3)$ $\oplus$ ${\cal D}_W$ ($W^{1}_{\delta}(\Omega_+,\CC^3)$ $\subset$ ${\cal D}$ $\subset$ $W^{1}_{-\delta}(\Omega_+,\CC^3)$) is equipped with the norm \eqref{eq15n12c2} where $\kappa =2T$ and $\gamma = \delta$.
From Lemma \ref{lemext} it is possible to extend the form $b^1_{\Omega_+}$ (still denoted $b^1_{\Omega_+}$) on the space ${\cal D}_W\times W^{1}_{-\delta}(\Omega_+,\CC^3)$ and thus to extend the form $b_{\Omega_+}$ (denoted $\tilde{b}_{\Omega_+}$) on the space ${\cal D}\times W^{1}_{-\delta}(\Omega_+,\CC^3)$ (in particular on the space ${\cal D}\times {\cal D}$) by: for $u= u_W + u_{\delta} \in {\cal D}$, with $u_W \in {\cal D}_W$, $u_{\delta} \in W^{1}_{\delta}(\Omega_+,\CC^3)$, for $v \in W^{1}_{-\delta}(\Omega_+,\CC^3)$
\begin{equation}
\label{eq43b}
\tilde{b}_{\Omega_+}(u,v)= b^1_{\Omega_+}(u_W,v) + b_{\Omega_+}(u_{\delta},v)
\end{equation}
and from Lemma \ref{lemext} there exists $C >0$ satisfying
\begin{equation}
\label{eq15n12d1}
|\tilde{b}_{\Omega_+}(u,v)| \leq C ||u||_{{\cal D}} ||v||_{-\delta,1,\Omega_+},\,u \in {\cal D}, v \in W^1_{-\delta}(\Omega_+,\CC^3).
\end{equation}
From \eqref{eq15f1a} for $u_W \in {\cal D}_W$ and $v \in W^{1}_{\delta}(\Omega_+,\CC^3)$, $b^1_{\Omega_+}(u_W,{v})= {b}_{\Omega_+}(u_W,{v})$ so that 
for $u \in {\cal D}$ and $v \in W^{1}_{\delta}(\Omega_+,\CC^3)$, $\tilde{b}_{\Omega_+}(u,{v})= {b}_{\Omega_+}(u,{v})$.

Moreover the sum $W^{1}_{\delta,\Sigma}(\Omega_+,\CC^3)$ $+$ ${\cal D}_W$ is direct and the space ${\cal D}_{\Sigma}$ = $W^{1}_{\delta,\Sigma}(\Omega_+,\CC^3)$ $\oplus$ ${\cal D}_W$ ($W^{1}_{\delta,\Sigma}(\Omega_+,\CC^3)$ $\subset$ ${\cal D}_{\Sigma}$ $\subset$ $W^{1}_{-\delta,\Sigma}(\Omega_+,\CC^3)$) is a closed subspace of ${\cal D}$.
From \eqref{eq15n12d1} it is possible to extend $B_{\Omega_+,\delta}$ (denoted $\tilde{B}_{\Omega_+,\delta}$) on ${\cal D}_{\Sigma}$ by:
%
%
\begin{equation}
\label{eq2010}
<\tilde{B}_{\Omega_+,\delta}u,v> = \tilde{b}_{\Omega_+}(u,\overline{v}),\,u \in {\cal D}_{\Sigma},\,v \in W^1_{-\delta,\Sigma}(\Omega_+,\CC^3),
\end{equation}
so that if $u \in {\cal D}_{\Sigma}$, $\tilde{B}_{\Omega_+,\delta}u$ $\in$ $(W^1_{-\delta,\Sigma}(\Omega_+,\CC^3))^*$ and there exists $C >0$ satisfying
\begin{equation}
\label{eq15n12d2}
||\tilde{B}_{\Omega_+,\delta}u||_{-\delta,\Sigma,1,\Omega_+,*} \leq C ||u||_{{\cal D}},\,u \in {\cal D}_{\Sigma}.
\end{equation}
Now define the sesquilinear form $q_{\Omega_+}$ on ${\cal D} \times {\cal D}$ by 
\begin{equation}
\label{eq15n11}
q_{\Omega_+}(u,v)= \tilde{b}_{\Omega_+}(u,v) - \overline{\tilde{b}_{\Omega_+}(v,u)}, \, u,v \in {\cal D}.
\end{equation}
Obviously if $u \in W^{1}_{\delta}(\Omega_+,\CC^3)$ or $v \in W^{1}_{\delta}(\Omega_+,\CC^3)$ then $q_{\Omega_+}(u,v)=0$. Consequently if $u$ and $v$ $\in {\cal D}$ are written under the form $u= u_W+ u_{\delta}$, $v= v_W+ v_{\delta}$ where $u_W$, $v_W$ $\in {\cal D}_W$, $u_{\delta}$, $v_{\delta}$ $\in W^{1}_{\delta}(\Omega_+,\CC^3)$, then
$q_{\Omega_+}(u,v)= q_{\Omega_+}(u_W,v_W)$.
The next lemma will be used several times in the sequel to extend some variational formulations to a larger space. It is to be compared with Lemma \ref{noyau1}.
\begin{lemma}
\label{noyau}
Assume that $f \in (W^{1}_{-\delta,\Sigma}(\Omega_+,\CC^3))^*$ and $u \in {\cal D}_{\Sigma}$ satisfy $B_{\Omega_+,-\delta} u = f\vert_{W^{1}_{\delta,\Sigma}(\Omega_+,\CC^3)}$, that is 
\begin{equation}
\label{eq1000}
b_{\Omega_+}(u,v) = \langle f,v\rangle, \, v \in W^{1}_{\delta,\Sigma}(\Omega_+,\CC^3).
\end{equation}
Then
\begin{equation}
\label{eq1001}
\tilde{b}_{\Omega_+}(u,v) = \langle f,v\rangle, \, v \in W^{1}_{-\delta,\Sigma}(\Omega_+,\CC^3),
\end{equation}
that is to say $\tilde{B}_{\Omega_+,\delta} u = f$.
In particular if $u \in {\cal D}_{\Sigma} \cap KerB_{\Omega_+, -\delta}$ then if $v \in W^{1}_{-\delta,\Sigma}(\Omega_+,\CC^3)$, $\tilde{b}_{\Omega_+}(u,v)=0$  consequently $\tilde{b}_{\Omega_+}(u,u)=0$ and $q_{\Omega_+}(u,u) =0$.
\end{lemma}
%
\begin{proof}
It is simply a consequence of \eqref{eq15n12d1}, \eqref{eq1000} and the density of $W^{1}_{\delta,\Sigma}(\Omega_+,\CC^3)$ in $W^{1}_{-\delta,\Sigma}(\Omega_+,\CC^3)$.
%
\end{proof}

The following proposition shows that the form $q_{\Omega_+}$ is nondegenerate. Its proof uses the same method as the proof of Theorem 2.1, p.50 of \cite{Nazarov-Plamenevsky} (roughly speaking a Taylor formula for $\mathscr{L}(\partial_x)$ around an eigenvalue).
\begin{proposition}
\label{prosympl}
Let $\nu_1 = i \omega_1$, $\ldots$, $\nu_{N} = i \omega_N$ be the eigenvalues of ${\mathscr{L}}(\nu)$ on the imaginary axis.
With the notations \eqref{eq15k1} and \eqref{eq15m} of Proposition \ref{pro2sol}, we have
\begin{equation}
\label{eq15n13}
q_{\Omega_+}(\chi(x_3) u_{j,s}^k, \chi(x_3) v_{j', s'}^{k'}) = \delta_{k,k'} \delta_{j,j'} \delta_{s+s'= \kappa_{k,j} -1}.
\end{equation}
This shows that the form $q_{\Omega_+}$ is nondegenerate.
\end{proposition}
%
\begin{proof}
For $i=1,\ldots,N$, $-\overline{\nu_i} = \nu_i$, therefore the set of "power-exponential" functions $v^i_{j,s}$ given in \eqref{eq15m} (as well as the "power-exponential" functions $u^i_{j,s}$ given in \eqref{eq15k1})
for $j=1,\ldots,J_i, s=0, \ldots,\kappa_{i,j} -1$ form a basis of ${\cal N}({\mathscr{L}}(\partial_x), \nu_i)$ = ${\cal N}({\mathscr{L}}(\partial_x), -\overline{\nu_i})$ (Proposition \ref{pro2sol}, (a)).
Let $\nu_k$ and $\nu_{k'}$, $1 \leq k,k' \leq N$ be two eigenvalues of ${\mathscr{L}}(\nu)$ on the imaginary axis, $u_{j,s}^k$, $v_{j', s'}^{k'}$ be power-exponential functions corresponding to $\nu_k$ and $\nu_{k'}$ given by \eqref{eq15k1} and \eqref{eq15m} and set $u = \chi(x_3) u_{j,s}^k$, $v= \chi(x_3) v_{j', s'}^{k'}$. If $R>0$ set $\Omega_{R}= \{ (x_1,x_3) \in \Omega_+, \, x_3 < R \}$ = $(-h,h) \times (0,R)$ and $d \Omega_{R}= \{ (x_1,x_3) \in \partial \Omega_+, \, x_3 < R \}$. Since $\partial_x \chi$ has a compact support  $\subset [1,2]$, it follows that for $R\geq 2$, $\tilde{b}_{\Omega_+}(u,v)$ = ${b}^1_{\Omega_+}(u,v)$ = ${b}^1_{\Omega_+,R}(u,v)$ where
\begin{equation}
\label{eq15f2}
{b}^1_{\Omega_+,R}(u,v)= \displaystyle{\int_{\Omega_{R}} ({\cal A}(\partial_{x_1}, \partial_{x_3}) u) \cdot {{v}} + \int_{d \Omega_{R}}({\cal B}(\partial_{x_1}, \partial_{x_3}) u) \cdot {{v}}}
\end{equation}
(with the notation \eqref{eq15g1a}) and that ${b}^1_{\Omega_+}(u,v)$ = $\underset{R \rightarrow + \infty}{\mbox{lim}} {b}^1_{\Omega_+,R}(u,v)$.
Owing to \eqref{eq15h}, \eqref{eq104d} and Green formula
\begin{eqnarray}
\label{eq15f3}
&&\displaystyle{{b}^1_{\Omega_+,R}(u,v)}= \displaystyle{\int_{0}^{R} 
\langle  \mathscr{L}(\partial_x)u(., x),{v}(.,x)\rangle_2 dx}= \nonumber\\
&&\displaystyle{\sum_{n=0}^2 \int_{0}^{R} 
\langle  \frac{\partial^n_x \mathscr{L}(\nu_k)}{n!} (\partial_x - \nu_k)^n u(., x),{v}(.,x)\rangle_2 dx}=\nonumber\\
&&\displaystyle{\sum_{n=0}^2 (-1)^n \int_{0}^{R} 
\langle  \frac{\partial^n_x \mathscr{L}(\nu_k)}{n!}  u(., x),({\partial_x - \nu_k)^n v}(.,x)\rangle_2 dx}+\nonumber\\
&&\displaystyle{\sum_{n=1}^2 
\sum_{\substack{p \in \NN, p' \in \NN,\\ p+p'=n-1}}(-1)^{p'}
\langle  \frac{\partial^n_x \mathscr{L}(\nu_k)}{n!}  (\partial_x- \nu_k)^p u(., R),{(\partial_x-\nu_k)^{p'} v}(.,R)\rangle_2}.
\end{eqnarray}
\eqref{eq104e2} implies 
\begin{equation}
\label{eq104e3}
\langle \partial_x^n{\mathscr L}(\nu_k) \varphi, \psi \rangle_2 =(-1)^n\overline{\langle \partial_x^n{\mathscr L}(\nu_k) \psi, \varphi \rangle_2},\, n\in \NN ,\, \varphi \in H^1(\omega_h, \CC^3), \psi \in H^1(\omega_h, \CC^3)
\end{equation}
therefore
\begin{eqnarray}
\label{eq15f4}
&&\displaystyle{\sum_{n=0}^2 (-1)^n \int_{0}^{R} 
\langle  \frac{\partial^n_x \mathscr{L}(\nu_k)}{n!}u(., x), {(\partial_x-\nu_k)^nv}(.,x) \rangle_2 dx}=\nonumber\\
&&\displaystyle{\sum_{n=0}^2 \int_{0}^{R} 
\overline{\langle  \frac{\partial^n_x \mathscr{L}(\nu_k)}{n!} {(\partial_x-\nu_k)^nv}(.,x), u(., x)\rangle_2} dx}=\nonumber\\
&&\displaystyle{\int_{0}^{R} 
\overline{\langle  \mathscr{L}(\partial_x)v(., x),{u}(.,x)\rangle_2} dx} =
\displaystyle{\overline{{b}^1_{\Omega_+,R}(v,u)}}.
\end{eqnarray}
On the other hand for $R \geq 2$, 
\begin{eqnarray}
\label{eq15f5}
&&\displaystyle{\sum_{n=1}^2 
\sum_{\substack{p \in \NN, p' \in \NN,\\ p+p'=n-1}}(-1)^{p'}
\langle  \frac{\partial^n_x \mathscr{L}(\nu_k)}{n!}  (\partial_x-\nu_k)^{p} u(., R),{(\partial_x-\nu_k)^{p'} v}(.,R)\rangle_2}=
\nonumber\\
&&\displaystyle{\sum_{n=1}^2 \sum_{\substack{p \in \NN, p' \in \NN,\\ p+p'=n-1}}(-1)^{p'}
\langle  \frac{\partial^n_x \mathscr{L}(\nu_k)}{n!}  (\partial_x-\nu_k)^p 
(e^{\nu_k x} \sum_{\sigma=0}^s \frac{x^{\sigma}}{\sigma !} \varphi_{j,s - \sigma}^k)(.,R)},
\nonumber\\
&&\displaystyle{(\partial_x-\nu_k)^{p'} 
(e^{\nu_{k'} x} \sum_{\sigma'=0}^{s'} \frac{(-x)^{\sigma'}}{\sigma' !} 
\psi_{j',s' - \sigma'}^{k'})(.,R)\rangle_2}=
\nonumber\\
&&\displaystyle{\sum_{n=1}^2 \sum_{\substack{p \in \NN, p' \in \NN,\\ p+p'=n-1}}(-1)^{p'} e^{(\nu_k - \nu_{k'}) R}
\langle  \frac{\partial^n_x \mathscr{L}(\nu_k)}{n!}
\partial_x^p 
(\sum_{\sigma=0}^s \frac{x^{\sigma}}{\sigma !} \varphi_{j,s - \sigma}^k)(.,R)},
\nonumber\\
&&\displaystyle{(\partial_x - \nu_k + \nu_{k'})^{p'} 
(\sum_{\sigma'=0}^{s'} \frac{(-x)^{\sigma'}}{\sigma' !} 
\psi_{j',s' - \sigma'}^{k'})(.,R)\rangle_2}.
\end{eqnarray}
Since the term in \eqref{eq15f5} has a finite limit when $R \rightarrow +\infty$ this limit is necessarily
\begin{eqnarray}
\label{eq15f6}
&&\displaystyle{\delta_{k,k'} \sum_{n=1}^2 \sum_{\substack{p \in \NN, p' \in \NN,\\ p+p'=n-1}}
\langle  \frac{\partial^n_x \mathscr{L}(\nu_k)}{n!}
\varphi_{j,s - p}^k},
\displaystyle{ 
\psi_{j',s' - p'}^{k'}\rangle_2}
\end{eqnarray}
and from \eqref{eq212}  this limits is $\delta_{k,k'} \delta_{j,j'} \delta_{s+s'= \kappa_{k,j} -1}$.
%
\end{proof}
\begin{remark}
\label{nondegenerate}
According to \cite{Nazarov-Plamenevsky}, Theorem 2.1, pp.148-156 (whose proof is rather long, about eight pages), and arguing as in the proof of Proposition \ref{prosympl}, for $k =1,\ldots, N$ there exists a canonical system of Jordan chains of ${\cal L}(\nu)$ corresponding to $\nu_k$, $\{\varphi_{j,s}^k\}_{j=1,\ldots,J_k, s=0,\ldots,\kappa_{k,j} -1}$ satisfying (see \cite{Nazarov-Plamenevsky}, p.157)
\begin{equation}
\label{eq15n12}
q_{\Omega_+}(\chi(x_3) u^k_{j,s}, \chi(x_3) u^{k'}_{j',s'})= \pm i\delta_{k,k'} \delta_{j,j'} \delta_{s+s' = \kappa_{k,j}-1}.
\end{equation}
This result is finer than Proposition \ref{prosympl}, but Proposition \ref{prosympl} is sufficient for our purpose.
There is another proof of \cite{Nazarov-Plamenevsky}, Theorem 2.1, pp.148-156 in \cite{Moller}, Theorem 3.1.
\end{remark}

The form $q_{\Omega_+}$ is sesquilinear and if $u,v \in {\cal D}$, $q_{\Omega_+}(v,u)$ = $-\overline {q_{\Omega_+}(u,v)}$ consequently $i q_{\Omega_+}$ is Hermitian. Moreover $q_{\Omega_+}$ is nondegenerate, thus it is symplectic (that is to say by definition sesquilinear, antihermitian and nondegenerate, see \cite{Everitt}, Definition 1) and consequently there exists a basis $V_1,\ldots, V_{2T}$ of ${\cal D}_W$ with the property that $i q_{\Omega_+}(V_j, V_k) = \pm \delta_{j,k}$.

From Eq. \eqref{eq15f2} and Green formula, it is easily seen that for $R \geq 2$
\begin{equation}
\label{eq1030}
q_{\Omega_+}(u,v) = \int_{\omega_h \times \{R\}} \sigma_{i3}(\overline{v}) {u}_i - \sigma_{i3}(u) \overline{v}_i, \, u,v \in {\cal D}_{W},
\end{equation}
where the second member of \eqref{eq1030} is independent of $R\geq 2$. The expression of $q_{\Omega_+}$ given by \eqref{eq1030} is exactly the opposite of the expression given by Eq. (3.5) of \cite{Nazarov}. By the Mandelstam radiation principle (see \cite{Nazarov}, section 3.1), the direction of wave propagation of $u$ $\in {\cal D}_W$ is determined by the sign of the projection onto the $x_3$-axis of the time-average on a period of the integral of the Umov-Poynting vector of the energy flux over a cross-section $\omega_h \times \{R\}$. This sign is the same as that of $i q_{\Omega_+}(u,u)$ (see \cite{Nazarov}, Eq. (3.4)). According to the Mandelstam radiation principle, $u \in {\cal D}_{W}$ is said to be outgoing (resp. incoming) if this sign is $>0$, i.e. $i q_{\Omega_+}(u,u)>0$ (resp. $<0$, i.e. $i q_{\Omega_+}(u,u)<0$).

The following lemma corresponds to Theorem 3.2, p.157 and Proposition 3.3, p.158 of \cite{Nazarov-Plamenevsky} in the present context. Its proof is very similar to the proofs of the aforementioned results and is given here for the sake of completeness.
\begin{lemma}
\label{scatt}
(a) There exists a basis $U_1,\ldots,U_{2T}$ of ${\cal D}_W$ satisfying
\begin{equation}
\label{eq15n12a}
q_{\Omega_+}(U_j,U_k)=  0 \mbox{ for } j\neq k,
\end{equation}
\begin{equation}
\label{eq15n12b}
q_{\Omega_+}(U_j,U_j)=  -i,\, q_{\Omega_+}(U_{j+T},U_{j+T})=  i \mbox{ for } j = 1,\ldots,T.
\end{equation}
$U_1,\ldots,U_{T}$ are outgoing and $U_{T+1},\ldots,U_{2T}$ are incoming.

(b) There exists a basis $\{\zeta_1,\ldots,\zeta_T\}$ and a basis $\{\eta_1,\ldots,\eta_T\}$ of a complement of $KerB_{\Omega_+,\delta}$ in $KerB_{\Omega_+,-\delta}$ such that for $i=1,\ldots,T$, 
\begin{equation}
\label{eq44}
\zeta_i - U_i -\sum_{k=1}^T t_{ik} U_{k+T} \in W^1_{\delta,\Sigma}(\Omega_+,\CC^3),
\end{equation}
\begin{equation}
\label{eq45}
\eta_i - U_{i+T} - \sum_{k=1}^T s_{ik} U_{k} \in W^1_{\delta,\Sigma}(\Omega_+,\CC^3),
\end{equation}
where $t_{ik}$, $s_{ik}$ $\in \CC$ ($i, k =1,\ldots,T$),  therefore $\zeta_i$, $\eta_i$ $\in {\cal D}_{\Sigma}$ ($i=1,\ldots,T$).
\end{lemma}
%
\begin{proof}
(a) Let $Z_1,\ldots,Z_T$ be a basis of a complement of $KerB_{\Omega_+,\delta}$ in $KerB_{\Omega_+,-\delta}$. 
Let us denote by $V^+_j$, $j =1,\ldots T^+$ (resp. $V^-_j$, $j =1,\ldots T^-$) the familly of $V_i$ satisfying $i q_{\Omega_+}(V^+_j,V^+_j) =1$ (resp. $i q_{\Omega_+}(V^-_j,V^-_j) = -1$), with $2T= T^+ + T^-$.
Proposition \ref{asymt} applied with $f=0$ shows that there exists a matrix $C^+ = (C^+_{ij})_{i=1,\ldots,T; j=1,\ldots, T^+}$ (resp. $C^- = (C^-_{ij})_{i=1,\ldots,T; j=1,\ldots, T^-}$) of complex numbers satisfying 
\begin{equation}
\label{eq1010}
Z_k - C^+_{kj} V^+_j - C^-_{kj} V^-_j\in W^1_{\delta,\Sigma}(\Omega_+,\CC^3),\, k=1,\ldots, T.
\end{equation}
Assume that $X \in \CC^T$ is such that $(C^+)^t X=0$ (where $(C^+)^t$ is the transpose matrix of $C^+$). Therefore $X_k Z_k - X_k C^-_{kj}V^-_j$ $\in W^1_{\delta,\Sigma}(\Omega_+,\CC^3)$ and $X_k Z_k \in {\cal D}_{\Sigma}\cap KerB_{\Omega_+,-\delta}$, and in view of Lemma \ref{noyau}, we get $q_{\Omega_+}(X_k Z_k,X_k Z_k) =0$ thus $q_{\Omega_+}(X_k C^-_{kj}V^-_j,X_k C^-_{kj}V^-_j)=0$ and $X_k C^-_{kj} =0$ ($j=1,\ldots,T^-$). Consequently $X_k Z_k \in W^1_{\delta,\Sigma}(\Omega_+,\CC^3)$ but given that $X_k Z_k \in KerB_{\Omega_+,-\delta}$, Lemma \ref{noyau1} implies that $X_k Z_k \in KerB_{\Omega_+,\delta}$ and $X_k=0$, $k=1,\ldots T$. This shows that necessarily $T^+ \geq T$. Likewise one shows that $ T^- \geq T$ and thus $T^+ = T^- =T$. Set $U_j= V_j^+$ for $j=1,\ldots, T$ and $U_{T+j}= V_j^-$ for $j=1,\ldots, T$. The result (a) follows. 

(b) Take $\zeta_i = ((C^+)^{-1})_{ik} Z_k$ and $\eta_i = ((C^-)^{-1})_{ik} Z_k$ where $Z_k$ satisfies Eq. \eqref{eq1010}.
%
\end{proof}

Consider ${\cal D}^{out}$ ($\subset {\cal D}$) the (topological) direct sum of $W^{1}_{\delta}(\Omega_+,\CC^3)$ and the vector space spanned by $U_1,\ldots,U_{T}$, equipped with the norm: for $u= \sum_{k=1}^T a_k U_k + u_{\delta}$  $\in {\cal D}^{out}$ with $a_1,\ldots,a_T \in \CC$ and $u_{\delta} \in W^{1}_{\delta}(\Omega_+,\CC^3)$,

\begin{equation}
\label{eq15n12c}
\displaystyle{||u||^2_{{\cal D}^{out}} = \sum_{k=1}^T |a_k|^2 + ||u_{\delta}||^2_{\delta,1,\Omega_+}}.
\end{equation}
${\cal D}^{out}$ is a Hilbert space for the inner product associated with this norm.
Thanks to \eqref{eq15n12d1} there exists $C >0$ satisfying
\begin{equation}
\label{eq15n12d}
|\tilde{b}_{\Omega_+}(u,v)| \leq C ||u||_{{\cal D}^{out}} ||v||_{-\delta,1,\Omega_+},\,u \in {\cal D}^{out}, v \in W^1_{-\delta}(\Omega_+,\CC^3).
\end{equation}
In the same way 
${\cal D}_{\Sigma}^{out}$ ($\subset {\cal D}_{\Sigma}$ and $\subset {\cal D}^{out}$) the (topological) direct sum of $W^{1}_{\delta,\Sigma}(\Omega_+,\CC^3)$ and the vector space spanned by $U_1,\ldots,U_{T}$,
is a closed subspace of ${\cal D}$.

%
Define the linear continuous map $B_{\Omega_+,\delta}^{out}$: ${\cal D}^{out}_{\Sigma}$ $\rightarrow$ $({W}^1_{-\delta,\Sigma}(\Omega_+,\CC^3))^*$ as $B_{\Omega_+,\delta}^{out}$ = $\tilde{B}_{\Omega_+,\delta}\vert_{{\cal D}^{out}_{\Sigma}}$ so that
\begin{equation}
\label{eq8a1}
<B_{\Omega_+,\delta}^{out}u,v> = \tilde{b}_{\Omega_+}(u,\overline{v}),\, u \in {\cal D}^{out}_{\Sigma}, v \in W^1_{-\delta,\Sigma}(\Omega_+,\CC^3)
\end{equation}
and
\begin{equation}
\label{eq9a1}
||B_{\Omega_+,\delta}^{out}u||_{-\delta,\Sigma,1,\Omega_+,*} \leq C ||u||_{{\cal D}^{out}}, \, u \in {\cal D}^{out}_{\Sigma}.
\end{equation}
\section{Well-posedness of the problem in the half-strip $\Omega_+$.}
\label{well-posedness}
\setcounter{equation}{0}
In this section we establish the well-posedness of the time-harmonic elasticity problem in a half-strip when the solution is searched in the space ${\cal D}^{out}$ (Theorem \ref{th-half-plane}).
In order to solve the problem in a half-strip with a non-homogeneous Dirichlet condition, the method consists in taking a trace lifting of this inhomogeneous Dirichlet condition and to reduce the problem to a problem with a homogeneous Dirichlet condition and a second member. The solution to this intermediate problem is given by the next theorem.
\begin{theorem}
\label{th-d-to-n} Under Assumption \ref{ass1}, the operator
$B_{\Omega_+,\delta}^{out}$: ${\cal D}^{out}_{\Sigma}$ $\rightarrow$ $({W}^1_{-\delta,\Sigma}(\Omega_+,\CC^3))^*$ is an isomorphism.
\end{theorem}
%
\begin{proof}
(a) According to \cite{Nazarov1}, Theorem 1 or \cite{Bouhennache}, Theorem 2.1, we know that $KerB_{\Omega_+,0} =\{0\}$, consequently $KerB_{\Omega_+,\delta} =\{0\}$. Let us examine $KerB_{\Omega_+,\delta}^{out}$. Assume that $u = \sum_{k=1}^T a_k U_k + u_{\delta} \in {\cal D}^{out}_{\Sigma}$ satisfies the condition $B_{\Omega_+,\delta}^{out} u=0$. Consequently $\tilde{b}_{\Omega_+}(u,u)=0$ and $q_{\Omega_+}(u,u)=0$. But $q_{\Omega_+}(u,u)= -i (\sum_{k=1}^T |a_k|^2)$ therefore $a_k=0$ ($k=1,\ldots, T$), $u = u_{\delta} \in {W}^1_{\delta,\Sigma}(\Omega_+,\CC^3)$ and $u \in KerB_{\Omega_+,\delta} =\{0\}$. We have shown that $KerB_{\Omega_+,\delta}^{out} =\{0\}$.

(b) $B_{\Omega+,-\delta}$ is a Fredholm operator, accordingly its range ${\cal R}(B_{\Omega+,-\delta})$ is closed and since $(B_{\Omega+,\delta})^*$ = $B_{\Omega+,-\delta}$, from  \cite{Kato}, Theorem 5.13, p.234, we get
\begin{equation}
\label{eq45a}
{\cal R}(B_{\Omega+,-\delta})= {\cal R}((B_{\Omega+,\delta})^*)= (Ker(B_{\Omega+,\delta}))^{\perp} = ({W}^1_{\delta,\Sigma}(\Omega_+,\CC^3))^*.
\end{equation}
Let $f \in ({W}^1_{-\delta,\Sigma}(\Omega_+,\CC^3))^*$. Hence $f\vert_{{W}^1_{\delta,\Sigma}(\Omega_+,\CC^3)} \in ({W}^1_{\delta,\Sigma}(\Omega_+,\CC^3))^*$ and due to \eqref{eq45a} there exists $u \in {W}^1_{-\delta,\Sigma}(\Omega_+,\CC^3)$ with the property that $B_{\Omega+,-\delta} u= f\vert_{{W}^1_{\delta,\Sigma}(\Omega_+,\CC^3)}$, namely
\begin{equation}
\label{eq45b}
b_{\Omega_+}(u, \overline{v}) = \langle f, v\rangle, \, v \in {W}^1_{\delta,\Sigma}(\Omega_+,\CC^3).
\end{equation}
Owing to Proposition \ref{asymt} and Lemma \ref{scatt}, (a), $u$ may be written under the form:
\begin{equation}
\label{eq45e}
u= \sum_{k=1}^{2T} a_k U_k + u_{\delta}
\end{equation}
where $a_1,\ldots,a_{2T} \in \CC$ and ${u}_{\delta} \in {W}^1_{\delta,\Sigma}(\Omega_+,\CC^3)$. Setting $u^{out} = u - \sum_{k=T+1}^{2T} a_k \eta_{k-T}$ where $\eta_k$ is given by \eqref{eq45} we obtain $u^{out} \in {\cal D}^{out}_{\Sigma}$ and $B_{\Omega_+,-\delta} u^{out} = f\vert_{{W}^1_{\delta,\Sigma}(\Omega_+,\CC^3)}$. But in view of Lemma \ref{noyau}, we obtain
 $B_{\Omega_+,\delta}^{out} u^{out} = f$.
\end{proof}
%

%
%
If $g \in H^{1/2}(\Sigma,\CC^3)$, let ${{\cal D}}^{out}_g$ be the set of $u \in {{\cal D}}^{out}$ such that $u\vert_{\Sigma}= g$: it is a closed (affine) subspace of ${{\cal D}}^{out}$ and ${{\cal D}}^{out}_0$ = ${\cal D}^{out}_{\Sigma}$.
%
%
If a solution to problem \eqref{eq2x1},\ldots,\eqref{eq4bx2} is searched in the space   
${{\cal D}}^{out}_g$, the corresponding variational formulation reads as follows: 
\begin{equation}
\label{eq45g1}
\mbox{Find }u \in {{\cal D}}^{out}_g\mbox{ such that for all } v \in W^1_{-\delta,\Sigma}(\Omega_+,\CC^3), \tilde{b}_{\Omega_+}(u,v) = 0.
\end{equation}
Remark first that a solution to \eqref{eq45g1} is necessarily unique because of part (a) in the proof of Theorem \ref{th-d-to-n}.
Now let us give $g \in H^{1/2}(\Sigma,\CC^3)$.
It is possible to find $u_0 \in H^1(\Omega_+,\CC^3)$ with compact support (so that $u_0 \in W^1_{\delta}(\Omega_+,\CC^3)$) in such a way that $u_0\vert_{\Sigma} = g$. 
Let us search a solution $u$ to \eqref{eq45g1} under the form $u =u_0 + w$ where $w \in {\cal D}_{\Sigma}^{out}$.
The variational formulation \eqref{eq45g1} boils down to the following one:
find  $w \in {\cal D}_{\Sigma}^{out}\mbox{ such that for all } v \in W^1_{-\delta,\Sigma}(\Omega_+,\CC^3),$
\begin{equation}
\label{eq45g}
\tilde{b}_{\Omega_+}(w,v) = -\tilde{b}_{\Omega_+}(u_0,v) =  -{b}_{\Omega_+}(u_0,v).
\end{equation}
The linear form $f$: $v \in W^1_{-\delta,\Sigma}(\Omega_+,\CC^3) \mapsto -b_{\Omega_+}(u_0,\overline{v})$ belongs to $(W^1_{-\delta,\Sigma}(\Omega_+,\CC^3))^*$ and \eqref{eq45g} reads as follows:
\begin{equation}
\label{eq45h}
B^{out}_{\Omega_+,\delta} w =f.
\end{equation}
Theorem \ref{th-d-to-n} shows that problem \eqref{eq45h} is well-posed. Now we state the main theorem of the paper.

\begin{theorem}
\label{th-half-plane}
Let $g \in H^{1/2}(\Sigma,\CC^3)$. Under Assumption \ref{ass1}, problem \eqref{eq45g1} has a unique solution $u$ $\in {{\cal D}}^{out}_g$ which is a weak solution to \eqref{eq2x1},\ldots,\eqref{eq4bx2}.
Moreover $u$ is written under the form $u=u_0+w$ where $u_0 \in W^1_{\delta}(\Omega_+, \CC^3)$ satisfies $u_{0}\vert_{\Sigma} = g$ and $w \in {\cal D}^{out}_{\Sigma}$ is written under the form
\begin{equation}
\label{eq45h1}
w= \sum_{k=1}^T a_k U_k + w_{\delta}
\end{equation}
where $a_k= -i b_{\Omega_+}(u_0,\zeta_k),\, k=1,\ldots,T$ and $w_{\delta} \in W^1_{\delta,\Sigma}(\Omega_+, \CC^3)$.
\end{theorem}
%
\begin{proof}
$w$ is written under the form \eqref{eq45h1}, therefore Eq. \eqref{eq45g} yields $\tilde{b}_{\Omega_+}(w,\zeta_k) = -{b}_{\Omega_+}(u_0,\zeta_k)$ ($k=1,\ldots,T$). But with notation \eqref{eq44}, for $k=1,\ldots,T$, $\zeta_k \in$ ${\cal D}_{\Sigma} \cap KerB_{\Omega+,-\delta}$, and due to Lemma \ref{noyau}, $\tilde{b}_{\Omega_+}(\zeta_k,w) = 0$ thus $q_{\Omega_+}(w,\zeta_k)= -{b}_{\Omega_+}(u_0,\zeta_k)$. On the other hand taking into account \eqref{eq45h1}, \eqref{eq15n12a}, \eqref{eq15n12b} and \eqref{eq44}, we obtain $q_{\Omega_+}(w,\zeta_k) = -i a_k$.
\end{proof}

\appendix
\section{Operator pencils}
\setcounter{equation}{0}
\label{oppenc}
In this section we recall some definitions and results about operator pencils (see \cite{Kozlov-Mazya-Rossmann_s}, Chapter 1, \cite{Kozlov-Mazya}, appendix A, \cite{Markus}, section 11, \cite{Kozlov-Mazya-Rossmann}, section 5.1, \cite{Gohberg-Krein}, Chapter V, section 9).
Let ${\cal X}$, ${\cal Y}$ be Hilbert spaces with respectively the inner products $(.,.)_{\cal X}$, $(.,.)_{\cal Y}$ and the norms $||.||_{\cal{X}}$,$||.||_{\cal{Y}}$. We denote by $L(\cal{X},\cal{Y})$ the set of linear and bounded operators from $\cal{X}$ into $\cal{Y}$. If $A \in L(\cal{X},\cal{Y})$, we denote the kernel and the range of the operator $A$ by $kerA$ and ${\cal{R}}(A)$. The operator $A$ is said to be Fredholm if ${\cal{R}}(A)$ is closed and the dimension of $kerA$ (= dim$kerA$) and the codimension of ${\cal{R}}(A)$ (= codim${\cal{R}}(A)$) are finite. In that case the index of $A$ is by definition Ind$A$ = dim$kerA$ - codim${\cal{R}}(A)$.

The operator polynomial
\begin{equation}
\label{eq200}
{\cal{U}}(\nu) = \sum_{k=0}^l A_k \nu^k, \, \nu \in \CC,
\end{equation}
where $A_k \in L(\cal{X},\cal{Y})$ is called a polynomial operator pencil.
The point $\nu_0 \in \CC$ is said to be regular if the operator ${\mathscr{U}}(\nu_0)$ is invertible. The set of all nonregular points is called the spectrum of the polynomial operator pencil ${\cal{U}}(\nu)$.

The number $\nu_0 \in \CC$  is called an eigenvalue of the operator
pencil ${\cal{U}}(\nu)$ if there exists $\varphi_0 \in {\cal X}$, $\varphi_0 \neq 0$, such that ${\cal{U}}(\nu_0)$$\varphi_0 = 0$. Such a vector $\varphi_0$ is called an eigenvector of ${\cal{U}}(\nu)$ corresponding to $\nu_0$.
The dimension of $ker{\cal{U}}(\nu_0)$ is called the geometric multiplicity of the eigenvalue $\nu_0$.

Let $\nu_0$ be an eigenvalue of the polynomial operator pencil ${\cal{U}}(\nu)$ and let $\varphi_0$
be an eigenvector of ${\cal{U}}(\nu)$ corresponding to $\nu_0$. If the elements $\varphi_1,\ldots, \varphi_{s-1} \in {\cal X}$ satisfy the equations
\begin{equation}
\label{eq201}
\sum_{k=0}^j \frac{1}{k!} {\cal{U}}^{(k)}(\nu_0) \varphi_{j-k} =0, \, j=1,\ldots,s-1,
\end{equation}
where ${\cal{U}}^{(k)} = d^k {\cal{U}} / d \nu^k$, the ordered collection  $\varphi_0,\ldots, \varphi_{s-1} \in {\cal X}$ is said to be a Jordan chain of ${\cal{U}}$ corresponding to the eigenvalue $\nu_0$. The number $s$ is called the length of this Jordan chain.
The vectors $\varphi_1,\ldots, \varphi_{s-1}$ are said to be generalized eigenvectors corresponding to $\varphi_0$.
The maximal length of all Jordan chains formed by the eigenvector $\varphi_0$ and
corresponding generalized eigenvectors will be denoted by $m(\varphi_0)$ and is called the rank of $\varphi_0$.
An eigenvalue $\nu_0$ of ${\cal U}(\nu)$ is said to be of finite algebraic multiplicity if the kernel of ${\cal U}(\nu_0)$ is finite dimensional and the ranks of the eigenvectors corresponding to $\nu_0$ have a common bound.
Assume that the eigenvalue $\nu_0$ of ${\cal U}(\nu)$ is of finite algebraic multiplicity and denote by $J$ its (finite) geometric multiplicity. In that case
\begin{equation}
\label{eq202}
\underset{\varphi \in  Ker{\cal U}(\nu_0) \setminus \{0\}}
{\mbox{ max } m(\varphi)} < + \infty.
\end{equation}
A set of Jordan chains
$\varphi_{j,0},\ldots, \varphi_{j,\kappa_j-1}$, $j=1,\ldots, J$
is called a canonical system of Jordan chains if

(a) the eigenvectors $\{\varphi_{j,0}\}_{j=1,\ldots,J}$ form a basis in $ker{\cal{U}}(\nu_0)$,

(b) for $j=1,\ldots, J$, let ${\cal M}_j$ be the space spanned by the vectors $\varphi_{1,0},\ldots, \varphi_{j-1,0}$. Then
\begin{equation}
\label{eq203}
m(\varphi_{j,0})= \underset{\varphi \in Ker{\cal U}(\nu_0) \setminus {\cal M}_j}{\mbox{max}} m(\varphi),\, j=1,\ldots,J.
\end{equation}
The numbers $\kappa_j = m(\varphi_{j,0})$ are called the partial multiplicities of the eigenvalue $\nu_0$.
The sum $\kappa_1 + \ldots + \kappa_J$ is called the algebraic multiplicity of the eigenvalue $\nu_0$.

The polynomial operator pencil ${\cal U}(\nu)$ is called Fredholm iff for $\nu \in \CC$ the operator ${\cal U}(\nu)$ is Fredholm and there exists at least one regular point for ${\cal U}(\nu)$ (see \cite{Kozlov-Mazya}, Definition A.8.1).
The spectrum of a Fredholm polynomial operator pencil consists of isolated eigenvalues of finite algebraic multiplicity (see \cite{Kozlov-Mazya}, Proposition A.8.4).
In particular if there exists at least one regular point $\nu_0$ for ${\cal U}(\nu)$ and the operators $A_k$, $k=1,\ldots,l$ are compact, 
for $\nu \in \CC$, ${\cal U}(\nu)$ is a compact perturbation of ${\cal U}(\nu_0)$ which is Fredholm with zero index, and \cite{McLean}, Theorem 2.26 shows that for $\nu \in \CC$, ${\cal U}(\nu)$ is Fredholm with zero index, accordingly the polynomial operator pencil ${\cal U}(\nu)$ is Fredholm and thus its spectrum consists of isolated eigenvalues of finite algebraic multiplicity.

Let ${\cal X}^*$ be the topological dual of the Hilbert space ${\cal X}$. Assume there is a duality between ${\cal X}$ and ${\cal X}^*$, i.e. a sesquilinear form $\langle .,. \rangle_1$ from ${\cal X} \times {\cal X}^*$ into $\CC$ satisfying (see \cite{Kozlov-Mazya}, p.404)
\begin{equation}
\label{eq204}
|\langle\varphi,\varphi^*\rangle_1| \leq ||\varphi||_{\cal X} ||\varphi^*||_{{\cal X}^*},\, \varphi \in {\cal X}, \, \varphi^* \in {{\cal X}^*},
\end{equation}
and such that if $h \in {{\cal X}^*}$, there exists a unique $\varphi^* \in {{\cal X}^*}$ satisfying
\begin{equation}
\label{eq205}
h(\varphi)= \langle\varphi,\varphi^*\rangle_1,\, \varphi \in {\cal X}.
\end{equation}
The same condition is assumed for the Hilbert space ${\cal Y}$, the corresponding sesquilinear form 
being denoted $\langle .,. \rangle_2$. If $A \in L({\cal X},{\cal Y})$, the adjoint operator $A^* \in L({\cal Y}^*,{\cal X}^*)$ satisfies the equality:
\begin{equation}
\label{eq206}
\langle\varphi,A^* \psi \rangle_1 = \langle A\varphi,\psi \rangle_2,\, \varphi \in {\cal X}, \psi \in {\cal Y}^*.
\end{equation}
Let ${\cal U}(\nu)$ be the polynomial operator pencil defined in \eqref{eq200}. The adjoint polynomial operator pencil ${\cal U}^*(\nu)$ is defined by (see \cite{Kozlov-Mazya}, appendix A.9)
\begin{equation}
\label{eq207}
{\cal{U}}^*(\nu) = ({\cal{U}}(\overline{\nu}))^* =  \sum_{k=0}^l A^*_k \nu^k, \, \nu \in \CC,
\end{equation}
that is to say ${\cal U}^*(\nu)$ satisfies the equality
\begin{equation}
\label{eq208}
\langle \varphi, {\cal{U}}^*(\nu) \psi\rangle_1 = \langle {\cal{U}}(\overline{\nu}) \varphi, \psi\rangle_2,\, \varphi \in {\cal X}, \, \psi \in {\cal Y}^*,\, \nu \in \CC.
\end{equation}
${\cal U}(\nu)$ is a Fredholm polynomial operator pencil iff ${\cal U}^*(\nu)$ is a Fredholm polynomial operator pencil (see \cite{Kozlov-Mazya}, p.422). 
$\nu_0$ is an eigenvalue of ${\cal U}(\nu)$ iff $\overline{\nu}_0$ is an eigenvalue of ${\cal U}^*(\nu)$ and the geometric, partial and algebraic multiplicities of $\nu_0$ and $\overline{\nu}_0$ coincide (see \cite{Kozlov-Mazya}, Proposition A.9.2). The following theorem is fundamental (see \cite{Kozlov-Mazya}, Theorem A.10.2):
\begin{theorem}
\label{thkeldysh}
Assume that ${\cal U}(\nu)$ defined in \eqref{eq200} is a Fredholm polynomial operator pencil and  $\nu_0$ is an eigenvalue of ${\cal U}(\nu)$ with geometric multiplicity $J$ and partial multiplicities $\kappa_1,\ldots,\kappa_J$. Assume that $\{\varphi_{j,s}\}_{j=1,\ldots,J, s=0,\ldots, \kappa_j-1}$ is a canonical system of Jordan chains corresponding to $\nu_0$. We have

(a) There exists a unique canonical system $\{\psi_{j,s}\}_{j=1,\ldots,J, s=0,\ldots, \kappa_j-1}$ of Jordan chains of ${\cal U}^*(\nu)$ corresponding to $\overline{\nu}_0$ such that $({\cal U}(\nu))^{-1}$ has the following representation
\begin{equation}
\label{eq210}
({\cal U}(\nu))^{-1}= \sum_{j=1}^J \sum_{s=0}^{\kappa_j -1}\frac{P_{j,s}}{(\nu - \nu_0) ^{\kappa_j -s}} + \Gamma(\nu)
\end{equation}
in a neighborhood of $\nu_0$, where $P_{j,s}$ $\in L({\cal Y},{\cal X})$ are defined by
\begin{equation}
\label{eq211}
P_{j,s} v = \sum_{\sigma = 0}^s \langle v, \psi_{j,\sigma} \rangle_2 \varphi_{j,s - \sigma}, v \in {\cal Y}
\end{equation}
and $\Gamma$ is a holomorphic operator function in a neighborhood of $\nu_0$ with values in $L({\cal Y},{\cal X})$.

(b) The canonical system $\{\psi_{j,s}\}_{j=1,\ldots,J, s=0,\ldots, \kappa_j-1}$ of the first assertion satisfies the biorthogonality conditions
\begin{equation}
\label{eq212}
\sum_{\substack{q,h,s \geq 0,\\
q+h+s=\kappa_j +n}} \frac{1}{q!} \langle {\cal U}^{(q)}(\nu_0)
\varphi_{j,h}, \psi_{k,s} \rangle_2 = \delta_{j,k} \delta_{0,n},\, j,k =1,\ldots,J, n=0,\ldots, \kappa_k-1.
\end{equation}
Here $\varphi_{j,h} =0$ for $h \geq \kappa_j$, $\psi_{k,s}=0$ for $s \geq \kappa_k$.
\end{theorem}

Let ${\cal U}(\nu)$ be a Fredholm polynomial operator pencil of the form \eqref{eq200}.
For $\nu \in \CC$ let ${\cal N}({\cal{U}}(\partial_x), \nu)$ be the set of solutions to the differential equation:
\begin{equation}
\label{eq15i}
{\cal{U}}(\partial_x)u = 0
\end{equation}
with the "power-exponential" form:
\begin{equation}
\label{eq15j}
u(x) = e^{\nu x} \sum_{\sigma=0}^s \frac{x^{\sigma}}{\sigma !} \varphi_{s - \sigma}
\end{equation}
where $s \in \NN$ is arbitrary, $\varphi_0$, $\ldots$, $\varphi_s$ $\in {\cal X}$ and $\varphi_0 \neq 0$. In that case the function \eqref{eq15j} is a solution to Eq.\eqref{eq15i} iff $\nu$ is an eigenvalue of ${\cal{U}}(\nu)$ and $\varphi_0$, $\ldots$, $\varphi_s$ is a Jordan chain corresponding to this eigenvalue (see \cite{Kozlov-Mazya-Rossmann}, Lemma 5.1.3).
Let $\nu$ be an eigenvalue of ${\cal{U}}(\nu)$, $J$ its geometric multiplicity, $\kappa_1,\ldots, \kappa_J$ its partial multiplicities, $\{\varphi_{j,s}\}_{j=1,\ldots,J, s=0,\ldots, \kappa_j -1}$ a canonical system of Jordan chains of ${\cal{U}}$ corresponding to $\nu$. It can be inferred that the "power-exponential" functions  
$\displaystyle{e^{\nu x} \sum_{\sigma=0}^s \frac{x^{\sigma}}{\sigma !} \varphi_{j,s - \sigma}}$, $j=1,\ldots,J$, $s=0,\ldots, \kappa_j -1$ form a basis of ${\cal N}({\cal{U}}(\partial_x), \nu)$. 
In particular the dimension of ${\cal N}({\cal{U}}(\partial_x), \nu)$ is equal to the algebraic multiplicity $\kappa_1 + \ldots + \kappa_J$ of the eigenvalue $\nu$ (see \cite{Kozlov-Mazya-Rossmann}, Lemma 5.1.4).
\section{Some linear independence results}
\label{results}
\setcounter{equation}{0}
In this section we establish results on the linear independence of some families of functions.
The present lemma is a preliminary result before Lemma \ref{indep}.
\begin{lemma}
\label{expol}
(a) Let $\omega_1,\ldots, \omega_N \in \RR$ be pairwise distinct, $P_1,\ldots, P_N$ be N polynomial functions with values in a Hilbert space ${\cal H}$, $c \in \RR$ and $f$ be the function defined on $[c, + \infty)$ by
\begin{equation}
\label{eq15n}
f(t) = \sum_{j=1}^N e^{i \omega_j t} P_j(t),\, t \in [c, + \infty).
\end{equation}
In that case $f \in L^2([c, + \infty), {\cal H})$ $\Rightarrow$ $P_j \equiv 0$ $(j =1,\ldots, N)$ $\Rightarrow$ $f\equiv 0$.

(b) More generally let $\omega_1,\ldots, \omega_N \in \RR$, $\beta_1,\ldots,\beta_N \in \RR_+$, $P_1,\ldots, P_N$ be N polynomial functions  with values in a Hilbert space ${\cal H}$, $c \in \RR$ and $f$ be the function defined on $[c, + \infty)$ by
\begin{equation}
\label{eq15na}
f(t) = \sum_{j=1}^N e^{(\beta_j +i \omega_j) t} P_j(t),\, t \in [c, + \infty).
\end{equation}
\end{lemma}
Then $f \in L^2([c, + \infty),{\cal H})$ $\Rightarrow$  $P_j \equiv 0$ $(j =1,\ldots, N)$ $\Rightarrow$ $f\equiv 0$.

\begin{proof}
We shall give a proof of this lemma when the Hilbert space ${\cal H}$ is the set of complex numbers $\CC$, the proof in the general case being similar.
(a) Assume there exists $1 \leq j\leq N$ with the property that $P_j \not \equiv 0$ and write $f$ under the form
\begin{equation}
\label{eq15o}
f(t) = \sum_{j=0}^M t^j Q_j(t),\, t \in [c,+ \infty),
\end{equation}
where $M \in \NN$,
\begin{equation}
\label{eq15p}
Q_j(t) = \sum_{k \in I_j} a_{j,k} e^{i \omega_k t},\, t \in [c,+ \infty),
\end{equation}
$I_j \subset \{1,\ldots, N \}$ ($j =0,\ldots, M$) and $Q_M \not \equiv 0$.
If $a$ and $b >0$, 
\begin{eqnarray}
\label{eq15q}
a^{2M} \int_{a}^{a+b} |Q_M(t)|^2 dt &\leq& \int_{a}^{a+b} |t^M Q_M(t)|^2 dt \nonumber \\
&\leq& 2 (\int_{a}^{a+b} |f(t)|^2 dt+ \int_{a}^{a+b} |\sum_{j=0}^{M-1} t^j Q_j(t)|^2 dt),
\end{eqnarray}
\begin{eqnarray}
\label{eq15r}
&&\displaystyle{\int_{a}^{a+b} |Q_M(t)|^2 dt }=  \displaystyle{b \sum_{k \in I_M} |a_{M,k}|^2 +  \sum_{k, l \in I_M, k \neq l} a_{M,k} \overline{a_{M,l}} \int_{a}^{a+b} e^{i (\omega_k - \omega_l)t} dt}=\nonumber \\
&&\displaystyle{b \sum_{k \in I_M} |a_{M,k}|^2} 
+  \displaystyle{\sum_{k, l \in I_M, k \neq l} a_{M,k} \overline{a_{M,l}} e^{i (\omega_k - \omega_l)(a+b/2)}(\frac{\sin b(\omega_k - \omega_l)/2}{(\omega_k - \omega_l)/2}}).
\end{eqnarray}
Owing to \eqref{eq15q} and \eqref{eq15r}, there exist $C_1>0$ and $C_2 >0$ satisfying
\begin{equation}
\label{eq15s}
b \sum_{k \in I_M} |a_{M,k}|^2 \leq
2 a^{-2M} (\int_{a}^{a+b} |f(t)|^2 dt+ C_1 b (a+ b)^{2(M-1)})+ C_2,\,a>0,b >0
\end{equation}
(in the case $M=0$, there is no second term in the parenthesis in \eqref{eq15s}).
In \eqref{eq15s}, first make $a$ tend to $+ \infty$ and next $b$ tend to $+ \infty$. We get
\begin{equation}
\label{eq15t}
\sum_{k \in I_M} |a_{M,k}|^2 = 0
\end{equation}
and therefore $Q_M \equiv 0$ in contradiction with the hypothesis. This shows that $Q_j\equiv0$ ($j =0,\ldots, M$) so that $P_j \equiv 0$ ($j=1,\ldots, N$) and $f \equiv 0$.

(b) Assume there exists $1 \leq j\leq N$ with the property that $P_j \not \equiv 0$ and write $f$ under the form
\begin{equation}
\label{eq15u}
f(t)= \sum_{p=1}^P e^{\beta_p t}(\sum_{k=0}^{M_p} t^k Q_{p,k}(t))
\end{equation}
where $1 \leq P \leq N$, $M_1,\ldots,M_P \in \NN$, $0 \leq \beta_1 < \cdots < \beta_P$,
\begin{equation}
\label{eq15v}
Q_{p,k}(t) = \sum_{l \in I_{p,k}} a_{p,k,l} e^{i \omega_l t},\, t \in [c,+ \infty),
\end{equation}
$I_{p,k} \subset \{1,\ldots, N \}$ ($p=1,\ldots,P$, $k =0,\ldots, M_p$), $\omega_l,l \in I_{p,k}$ being pairwise distinct, and $Q_{P,M_P} \not \equiv 0$.
By an inequality similar to \eqref{eq15q} where $Q_M$ in \eqref{eq15q} is replaced by $Q_{P,M_P}$, there exist $C_1>0$ and $C_2>0$ satisfying
\begin{eqnarray}
\label{eq15w}
&b& \sum_{l \in I_{P,M_P}} |a_{P,M_P,l}|^2 \leq \nonumber \\
&2& a^{-2M_P} e^{-2a \beta_P}(\int_{a}^{a+b} |f(t)|^2 dt+ C_1 b c(a,b))+ C_2, a >0, b>0,
\end{eqnarray}
where
\begin{equation}
\label{eq15x}
c(a,b) = \sup ((a+ b)^{2(M_P-1)}e^{2(a+ b) \beta_P}, (a+ b)^{2M_P}e^{2(a+b) \beta_{P-1}}).
\end{equation}
The rest of the proof is as in part (a).
\end{proof}

The next lemma shows the linear independence of some families of functions and enables to determine the dimension of the space ${\cal D}_W$ in section \ref{Asymptotics}.
\begin{lemma}
\label{indep}
Let $\beta < \gamma$ $\in \RR$ be such that $\mathscr{L}(\nu)$ has no eigenvalue on the lines $Re \nu = - \beta$ and $Re \nu = - \gamma$.
Let $\kappa$ be the sum of the algebraic multiplicities of the eigenvalues of $\mathscr{L}(\nu)$ (strictly) between the lines $Re \nu = - \beta$ and $Re \nu = - \gamma$. 
Let $\chi \in {\cal C}^{\infty}(\RR,\RR)$ be such that $\chi(x) =0$ for $x \leq 1$, $\chi(x) =1$ for $x \geq 2$.
With the notations of Proposition \ref{pro2sol}, (a), denote by $\{W_k\}_{k=1,\ldots,\kappa}$ the family $\{\chi(x_3) u_{j,s}^i\}_{i=1,\ldots, N,j=1,\ldots,J_i, s=0,\ldots,\kappa_{i,j} -1}$.
The functions $W_k$, ${k=1,\ldots,\kappa}$ are linearly independent modulo $W^0_{\gamma}(\Omega_+, \CC^3)$ (i.e., if $\alpha_1,\ldots, \alpha_{\kappa}$ $\in$ $\CC$ are such that $\sum_{k=1}^{\kappa} \alpha_k W_k \in W^0_{\gamma}(\Omega_+, \CC^3)$, then $\alpha_1= \cdots = \alpha_{\kappa} =0$). In particular they are linearly independent.
\end{lemma}
%
\begin{proof}
The lemma is a consequence of Lemma \ref{expol} and the fact that if $\nu_i$ is an eigenvalue of $\mathscr{L}(\nu)$, the "power-exponential" functions  
$\displaystyle{e^{\nu_i x} \sum_{\sigma=0}^s \frac{x^{\sigma}}{\sigma !} \varphi^i_{j,s - \sigma}}$, $j=1,\ldots,J_i$, $s=0,\ldots, \kappa_{i,j} -1$ form a basis of ${\cal N}(\mathscr{L}(\partial_x), \nu_i)$ (see appendix \ref{oppenc}). 
\end{proof}
\section{Result on the algebraic multiplicities}
\label{alg}
\setcounter{equation}{0}
The following result is similar to that of \cite{Nazarov-Plamenevsky}, p.111 or \cite{Baskin}, Theorem 2.2.4, but in the present context and with a distinct proof which is a little more "constructive".
\begin{proposition}
\label{ker*}
Let $\beta < \gamma$ $\in \RR$ be such that $\mathscr{L}(\nu)$ has no eigenvalue on the lines $Re \nu = - \beta$, $Re \nu = - \gamma$, $Re \nu = \beta$ and $Re \nu = \gamma$, consequently $B_{\Omega_+,\beta}$, $B_{\Omega_+,\gamma}$, $B_{\Omega_+,\beta}^*$ and $B_{\Omega_+,\gamma}^*$ are Fredholm operators (Proposition \ref{propfredholm}). Let $\kappa$ be the sum of the algebraic multiplicities of the eigenvalues of $\mathscr{L}(\nu)$ (strictly) between the lines $Re \nu = - \beta$ and $Re \nu = - \gamma$. Then $KerB_{\Omega_+,\gamma}$ $\subset$ $KerB_{\Omega_+,\beta}$, $KerB_{\Omega_+,\beta}^*$ $\subset$ $KerB_{\Omega_+,\gamma}^*$ and
\begin{equation}
\label{eq43a1}
dim (KerB_{\Omega_+,\beta}/KerB_{\Omega_+,\gamma})+dim (KerB_{\Omega_+,\gamma}^*/KerB_{\Omega_+,\beta}^*) = \kappa.
\end{equation}
\end{proposition}
%
\begin{proof}
The inclusions in the proposition are trivial. Set $d$ = $dim (KerB_{\Omega_+,\beta}/KerB_{\Omega_+,\gamma})$ and let $Z_1,\ldots, Z_d$ be a basis of a complement of $KerB_{\Omega_+,\gamma}$ in $KerB_{\Omega_+,\beta}$. Proposition \ref{asymt} applied with $f=0$ implies that there exists a matrix $C = (C_{ij})_{i=1,\ldots,d; j=1,\ldots, \kappa}$ of complex numbers satisfying the condition $Z_i - C_{ij} W_j \in W^1_{\gamma,\Sigma}(\Omega_+,\CC^3)$, $i=1,\ldots, d$. 
Assume that $X \in \CC^d$ satisfies $C^t X=0$ (where $C^t$ is the transpose matrix of $C$). 
Consequently $X_i Z_i$ = $X_i Z_i - X_i C_{ij}W_j$ $\in W^1_{\gamma,\Sigma}(\Omega_+,\CC^3)$.
But Lemma \ref{noyau1}, (b) shows that if $u \in KerB_{\Omega_+,\beta}$ and $u \in W^1_{\gamma,\Sigma}(\Omega_+,\CC^3)$ then $u \in KerB_{\Omega_+,\gamma}$, thus $X_i Z_i$ $\in KerB_{\Omega_+,\gamma}$ and $X=0$. This shows that the rank of the matrix $C$ is $d$. It follows that $\kappa \geq d$ and by changing the basis $Z_1,\ldots, Z_d$ and the order of $W_1,\ldots,W_{\kappa}$ if necessary, we may assume 
\begin{equation}
\label{eq43a2}
Z_i - W_i -\sum_{j=d+1}^{\kappa} C_{ij} W_j \in W^1_{\gamma,\Sigma}(\Omega_+,\CC^3), \, i=1,\ldots, d.
\end{equation}
Consider the space $E$ = $B_{\Omega_+,\beta}(W^1_{\beta,\Sigma}(\Omega_+,\CC^3))$ $\cap$ $i_{\beta,\gamma}(W^1_{-\gamma,\Sigma}(\Omega_+,\CC^3))^*$ (where $i_{\beta,\gamma}$ is defined in \eqref{eq300}). If $f \in E$, there exists $u_{\beta} \in W^1_{\beta,\Sigma}(\Omega_+,\CC^3)$ with the property that $f = B_{\Omega_+,\beta} u_{\beta} \in (W^1_{-\beta,\Sigma}(\Omega_+,\CC^3))^*$ and $f \in i_{\beta,\gamma}(W^1_{-\gamma,\Sigma}(\Omega_+,\CC^3))^*$. Proposition \ref{asymt} can be applied to $f$ and $u_{\beta}$. Applying  the operator $B_{\Omega_+,\beta}$ to Eq. \eqref{eq15l2} of this proposition and using \eqref{eq43a2} we obtain
\begin{equation}
\label{eq15l4}
f =  B_{\Omega_+,\beta} u'_{\gamma} + \sum_{i=d+1}^{\kappa} e_i (B_{\Omega_+,\beta} W_i)
\end{equation}
where $u'_{\gamma} \in W^1_{\gamma,\Sigma}(\Omega_+,\CC^3)$ and $e_{d+1},\ldots, e_{\kappa} \in \CC$.
Let $F$ be the vector space $F = B_{\Omega_+,\beta}(W^1_{\gamma,\Sigma}(\Omega_+,\CC^3))$ and $G$ be the vector space spanned by the functions $B_{\Omega_+,\beta} W_{i}$, $i=d+1,\ldots, \kappa$. We have shown $E \subset F+G$. 
Assume that $f \in E \cap G$. Consequently $f$ is written under the form $f$ = $B_{\Omega_+,\beta} v_{\gamma}$ = $\sum_{i=d+1}^{\kappa} f_i (B_{\Omega_+,\beta} W_i)$ where $v_{\gamma}$ $\in$ $W^1_{\gamma,\Sigma}(\Omega_+,\CC^3)$ and $f_{d+1},\ldots, f_{\kappa} \in \CC$. Thus $\sum_{i=d+1}^{\kappa} f_i W_i - v_{\gamma} \in$ $KerB_{\Omega_+,\beta}$, therefore we obtain a relation $\sum_{i=d+1}^{\kappa} f_i W_i - v'_{\gamma}$ = $\sum_{i=1}^d f_iZ_i$ where $f_{1},\ldots, f_{d} \in \CC$ and $v'_{\gamma}$ $\in$ $W^1_{\gamma,\Sigma}(\Omega_+,\CC^3)$. From \eqref{eq43a2} it follows that $\sum_{i=1}^d f_iW_i$ + $\sum_{i=d+1}^{\kappa} (\sum_{j=1}^d C_{ji} f_j - f_i) W_i$ $\in$ $W^1_{\gamma,\Sigma}(\Omega_+,\CC^3)$. 
Lemma \ref{indep} implies that $f_1 = \ldots f_{\kappa} =0$ thus $f=0$. 
We have shown that the sum $F+G$ is direct and also that the vectors $B_{\Omega_+,\beta} W_i$, $i=d+1,\ldots, \kappa$ are linearly independent so that $G$ is $\kappa - d$ dimensional.
On the other hand thanks to Lemma \ref{noyau1}, (a), $F$ $\subset$ $E$ and from Lemma \ref{lemext} and \eqref{eq15f1a}, $G$ $\subset$ $E$. Accordingly $E$ = $F$ $\oplus$ $G$.

We shall need an intermediate result. Assume more generally that ${\cal H}$ is a Hilbert space, that ${\cal E}$, ${\cal F}$, ${\cal G}$ are subspaces of ${\cal H}^*$ (the topological dual of ${\cal H}$) such that ${\cal E}$ = ${\cal F}$ $\oplus$ ${\cal G}$, that ${\cal G}$ is finite-dimensional and ${\cal F}$ is closed in ${\cal H}^*$ (so that ${\cal E}$  is closed in ${\cal H}^*$). Let $\{\varphi_i\}_{i=1,\ldots,q}$ ($q \in \NN^*$) be a basis of ${\cal G}$ and $\{\psi_i\}_{i=1,\ldots,q}$ be a family of vectors of ${\cal H}$ with the properties that $\langle\varphi_i, \psi_j\rangle = \delta_{i,j}$, $i,j =1,\ldots, q$ and $\psi_i \in {\cal F}^{\perp}$, $i =1,\ldots, q$ 
(here $\langle.,. \rangle$ is the duality bracket between ${\cal H}^*$ and ${\cal H}$ and if $A$ is a subset of ${\cal H}$ (resp. ${\cal H}^*$), $A^{\perp}$ is the orthogonal of $A$  in ${\cal H}^*$ (resp. ${\cal H}$) with respect to this duality bracket).
Let ${\cal G}_1 \subset {\cal H}$ be the space spanned by the family of vectors $\{\psi_i\}_{i=1,\ldots,q}$ (of dimension $q$). Let us show that  ${\cal F}^{\perp} = {\cal E}^{\perp} \oplus {\cal G}_1$. First ${\cal E}^{\perp}$ $\subset$ ${\cal F}^{\perp}$ and ${\cal G}_1$ $\subset$ ${\cal F}^{\perp}$. On the other hand ${\cal E}^{\perp} \cap {\cal G}_1$ $\subset$ ${\cal G}^{\perp} \cap {\cal G}_1 = \{0\}$. Finally if $x \in {\cal F}^{\perp}$, write $x =x_1 + x_2$ where $x_1 = x - \sum_{i=1}^q \langle \varphi_i,x\rangle \psi_i$. Therefore $x_1 \in {\cal F}^{\perp} \cap {\cal G}^{\perp} = {\cal E}^{\perp}$ and $x_2 = x-x_1 = \sum_{i=1}^q \langle \varphi_i,x\rangle \psi_i$ $\in$ ${\cal G}_1$. The result is proved. 

Apply this result to the Hilbert space ${\cal H}=W^1_{-\gamma,\Sigma}(\Omega_+,\CC^3)$ and to the spaces ${\cal E}=i_{\beta,\gamma}^{-1}E$, ${\cal F}=i_{\beta,\gamma}^{-1}F$ = $i_{\beta,\gamma}^{-1}B_{\Omega_+,\beta}(W^1_{\gamma,\Sigma}(\Omega_+,\CC^3))$ and ${\cal G}=i_{\beta,\gamma}^{-1}G$. Since $i_{\beta,\gamma}$ is injective and $E \subset i_{\beta,\gamma}(W_{-\gamma,\Sigma}^1(\Omega_+,\CC^3))^*$, it follows that ${\cal E} = {\cal F} \oplus {\cal G}$. Let ${\cal G}_1 \subset {\cal H}$ be the corresponding space of the previous construction of the intermediate result (of dimension $\kappa-d$). According to Lemma \ref{noyau1}, (b), ${\cal F}$ is exactly $B_{\Omega_+,\gamma}(W^1_{\gamma,\Sigma}(\Omega_+,\CC^3))$ hence closed in ${\cal H}^*$ (because  $B_{\Omega_+,\gamma}$ is Fredholm). Moreover ${\cal G}$ is finite dimensional of dimension $\kappa - d$, thereby closed in ${\cal H}^*$.
Remark that $B_{\Omega_+,\beta}(W^1_{\beta,\Sigma}(\Omega_+,\CC^3))$ is the orthogonal of $KerB_{\Omega_+,\beta}^*$ in $(W^1_{-\beta,\Sigma}(\Omega_+,\CC^3))^*$ so that $E$ =  $B_{\Omega_+,\beta}(W^1_{\beta,\Sigma}(\Omega_+,\CC^3))$ $\cap$ $i_{\beta,\gamma}(W^1_{-\gamma,\Sigma}(\Omega_+,\CC^3))^*$ is the orthogonal of $KerB_{\Omega_+,\beta}^*$ in $i_{\beta,\gamma}(W^1_{-\gamma,\Sigma}(\Omega_+,\CC^3))^*$ consequently ${\cal E}=i_{\beta,\gamma}^{-1}E$  is the orthogonal of $KerB_{\Omega_+,\beta}^*$ ($\subset W^1_{-\gamma,\Sigma}(\Omega_+,\CC^3)$) in $(W^1_{-\gamma,\Sigma}(\Omega_+,\CC^3))^*$ and ${\cal E}^{\perp}$ (the orthogonal of ${\cal E}$ in $W^1_{-\gamma,\Sigma}(\Omega_+,\CC^3)$) is exactly $KerB_{\Omega_+,\beta}^*$ (because $KerB_{\Omega_+,\beta}^*$ is finite-dimensional thus closed in $W^1_{-\gamma,\Sigma}(\Omega_+,\CC^3)$).
Finally ${\cal F}^{\perp}$(the orthogonal of ${\cal F}$ in $W^1_{-\gamma,\Sigma}(\Omega_+,\CC^3)$) is the orthogonal of $B_{\Omega_+,\gamma}(W^1_{\gamma,\Sigma}(\Omega_+,\CC^3))$ in $W^1_{-\gamma,\Sigma}(\Omega_+,\CC^3)$, that is $KerB_{\Omega_+,\gamma}^*$. We have shown
\begin{equation}
\label{eq1020}
KerB_{\Omega_+,\gamma}^* = KerB_{\Omega_+,\beta}^* \oplus {\cal G}_1.
\end{equation}
The proposition follows.
\end{proof}
%
%

%
\end{document}